\documentclass[12pt]{amsart}
\usepackage{mathrsfs}
\usepackage{bbm}
\usepackage{bm}
\usepackage{bbding}
\usepackage[mathcal]{euscript}
\usepackage{graphicx}
\usepackage{amsthm,amscd}
\usepackage{appendix}
\usepackage{enumerate}
\usepackage{calc}
\usepackage{mathtools}
\usepackage{dsfont}
\usepackage{CJK,fancyhdr}
\usepackage[dvipsnames]{xcolor}
\usepackage{anysize}
\usepackage{amssymb}
\usepackage{indentfirst}
\usepackage{tikz-cd}
\usepackage{lineno}
\usepackage[all,poly,knot]{xy}
\usepackage{xpatch}
\usepackage[pagebackref]{hyperref}
\tikzset{
	tail reversed/.code={\pgfsetarrowsstart{tikzcd to}}
}

\hypersetup{
	colorlinks=true,
	linkcolor=Blue,
	filecolor=YellowOrange,
	citecolor = WildStrawberry,      
	urlcolor=cyan,
}

\allowdisplaybreaks[2]

\providecommand{\U}[1]{\protect\rule{.1in}{.1in}}

\numberwithin{equation}{section}

\newtheorem{theorem} {Theorem} [section]
\newtheorem{proposition}[theorem]{Proposition}
\newtheorem{corollary}  [theorem]     {Corollary}
\newtheorem{lemma}  [theorem]     {Lemma}
\theoremstyle{definition}
\newtheorem{example}  [theorem]     {Example}

\newtheorem{step}{Step}
\xpatchcmd{\step}{%
	\normalfont\scshape\centering}{%
	\normalfont\scshape}{}{}%
\theoremstyle{definition}

\newtheorem{definition}  [theorem]     {Definition}

\newtheorem{conjecture}  [theorem]     {Conjecture}

\newtheorem*{thm A*}{Theorem A}
\newtheorem*{thm B*}{Theorem B}
\newtheorem*{thm C*}{Theorem C}
\newtheorem*{thm D*}{Theorem D}

\theoremstyle{definition}
\newtheorem{remark}  [theorem]     {Remark}

\newcommand{\im}{\mathrm{im}}

\newcommand{\cF}{\mathcal{F}}
\newcommand{\cO}{\mathcal{O}}

\newcommand{\CC}{\mathbb{C}}

\newcommand{\btheorem}{\begin{theorem}}
	\newcommand{\etheorem}{\end{theorem}}
\newcommand{\bproposition}{\begin{proposition}}
	\newcommand{\eproposition}{\end{proposition}}
\newcommand{\bdefinition}{\begin{definition}}
	\newcommand{\bbeta}{\boldsymbol{\beta}}	\newcommand{\edefinition}{\end{definition}}
\newcommand{\bcorollary}{\begin{corollary}}
	\newcommand{\ecorollary}{\end{corollary}}
\newcommand{\bproof}{\begin{proof}}
	\newcommand{\eproof}{\end{proof}}
\newcommand{\beq}{\begin{equation}}
	\newcommand{\eeq}{\end{equation}}
\newcommand{\ee}{\end{eqnarray*}}
\newcommand{\be}{\begin{eqnarray*}}

\newcommand{\elemma}{\end{lemma}}
\newcommand{\blemma}{\begin{lemma}}

\newcommand{\bd}{\begin{enumerate} }
	\newcommand{\ed}{\end{enumerate}}

\DeclareMathOperator{\Tor}{Tor}
\DeclareMathOperator{\codim}{codim}
\DeclareMathOperator{\Ann}{Ann}
\DeclareMathOperator{\Supp}{Supp}
\DeclareMathOperator{\Pic}{Pic}
\DeclareMathOperator{\BN}{BN}
\DeclareMathOperator{\BS}{BS}

\usepackage{adjustbox}
\usepackage{float}
\begin{document}
	\setcounter{tocdepth}{1}
	\title{Aspects of deformation invariance of plurigenera of K\"ahler varieties}
	\author{Jiawei Feng}
	\address{Department of Mathematics, Wuhan University, Hubei, Wuhan 430074, China}
	\email{jiaweifeng.math@gmail.com}
	\author{Yi Li}
	\address{Department of Mathematics, Wuhan University, Hubei, Wuhan 430074, China}
	\email{yilimath@whu.edu.cn}
	\author{Sheng Rao}
	\address{Department of Mathematics, Wuhan University, Hubei, Wuhan 430074, China}
	\email{likeanyone@whu.edu.cn}
	\begin{abstract}
		We study several aspects of deformation invariance of plurigenera in K\"ahler families over a
		smooth curve. First, for a relatively flat coherent sheaf, we characterize local constancy of
		fiberwise cohomology dimension by equivalent conditions in terms of extension, base
		change, torsion freeness of higher direct images, and the vanishing of
		jumping classes. Then, combining this criterion with the K\"ahler minimal model
		program, we prove invariance of plurigenera
		for K\"ahler families of threefolds whose central fiber has canonical
		singularities, and for smooth K\"ahler families of maximal Albanese
		dimension. We also establish deformation invariance of log plurigenera when
		the central fiber is projective, and is of log
		general type. We prove the invariance of higher cohomological plurigenera 
		assuming that the canonical bundle of the total space is semipositive.
		Finally, we extend Z.~Li--Z.~Wang's deformation invariance result for
		generalized polarized pairs with abundant nef parts from projective families
		to Moishezon families.
	\end{abstract}
	\maketitle
	\tableofcontents
	
	\section{Introduction}
	Plurigenera are among the fundamental birational invariants of compact
	complex manifolds, encoding the growth of their pluricanonical sections. A central problem in complex algebraic and analytic geometry is to determine whether plurigenera remain constant in smooth
	families. Let $f:\mathcal{X}\to T$ be a smooth  family of
	compact complex manifolds over a connected base, and write
	$X_t=f^{-1}(t)$ for $t\in T$. For $m\geq1$, the \emph{$m$-th plurigenus} or \emph{$m$-genus} of $X_t$ is
	\[
	P_m(X_t):=\dim H^0\!\left(X_t,mK_{X_t} \right).
	\]
	The problem is to determine when $P_m(X_t)$ is locally constant in $t$. 
	Kodaira--Spencer \cite{KS58} proved, by harmonic theoretic
	methods, that the Hodge numbers are locally constant in smooth K\"ahler
	families. For plurigenera, the decisive breakthrough was
	Y.-T.~Siu's proof~\cite{Siu98} of invariance for
	smooth projective families of manifolds of general type. Siu's
	argument introduced $L^2$ extension methods into the subject, combining the Ohsawa--Takegoshi extension 
	theorem \cite{OT87} with multiplier ideal sheaves
	\cite{Nad90}. Siu \cite{Siu02} subsequently removed the general type assumption.
	
	Siu's work led to many important refinements and extensions, such as Y.  Kawamata~\cite[Theorem~6]{KawaDeform99} and N. Nakayama~\cite[Chapter~VI]{Ny04}. M.~P{\u{a}}un~\cite{Paun07} gave a one-tower proof, while
	S.~Takayama~\cite{Taka06} treated Moishezon families and allowed singular
	fibers. S.~Rao--I.-H.~Tsai~\cite{RT1,RT2} obtained further results for
	fiberwise Moishezon families by means of bimeromorphic embeddings. 
	Recently, M.-L. Li--Rao--K. Wang--M. Wang \cite[Corollary~A.2]{Ra26} prove deformation invariance of plurigenera for a smooth family with uncountably many Moishezon fibers.

	From the
	viewpoint of the relative minimal model program, Nakayama~\cite{Naka85,Naka87}
	showed that the relative minimal model and abundance conjectures imply deformation
	invariance of plurigenera.
	C.~D.~Hacon--J.~M\textsuperscript{c}Kernan--C.~Xu~\cite{HMX13}
	proved deformation invariance of log plurigenera for log smooth families.
	For generalized pairs, S.~Filipazzi--R.~Svaldi~\cite[Theorem~4.1]{FS20}
	established deformation invariance of plurigenera for big and nef klt
	generalized pairs, while Z.~Li--Z.~Wang~\cite{LW26} proved an invariance
	theorem for generalized polarized pairs with abundant nef parts.
	
	Beyond the projective setting, the invariance of plurigenera for arbitrary
	smooth families of compact K\"ahler manifolds is still open:
	\begin{conjecture}[{\cite[Conjecture~0.4]{Siu02}}]\label{siu-conj}
		\textit{Let $\pi:X\to\Delta$ be a smooth family of compact K\"ahler manifolds over
			the unit disk, with fiber $X_t$. Then, for every $m\geq1$, the $m$-genus $P_m(X_t)$ is independent of $t\in\Delta$.}
	\end{conjecture}

	In this paper, we study this problem and related deformation invariance questions for plurigenera, higher
	cohomological plurigenera, log plurigenera, and generalized plurigenera. For
	a compact complex manifold $X$, we set
	\[
	P_m^{(k)}(X):=\dim H^k\!\left(X,\mathcal O_X(mK_X)\right),
	\qquad m\geq1,\quad k\geq0.
	\]
	Thus $P_m^{(0)}(X)=P_m(X)$, and we refer to the collection
	$\{P_m^{(k)}(X)\}_{k\geq0}$ as the \emph{higher cohomological
		plurigenera}. For a log smooth pair $(X,B)$ and an integer $m\geq1$ such
	that $mB$ is integral, we define the \emph{log plurigenus} by
	\[
	P_m(X,B):=\dim H^0\!\left(X,
	\mathcal O_X\bigl(m(K_X+B)\bigr)\right).
	\]
	For a $\mathbb Q$-divisorial generalized pair $(X,B+\mathbf M)$, we define
	the \emph{generalized plurigenus}, whenever
	$m(K_X+B+\mathbf M_X)$ is Cartier, by
	\[
	P_m(X,B+\mathbf M):=\dim H^0\!\left(X,
	\mathcal O_X\bigl(m(K_X+B+\mathbf M_X)\bigr)\right).
	\]
	When $B=0$, we abbreviate this as $P_m(X,\mathbf M)$.

	Our starting point is a cohomological criterion for proper morphisms over a
	smooth Stein curve. For a coherent sheaf flat over the base, it characterizes the local constancy of fiberwise cohomology dimensions in terms of cohomological flatness, base change and
	extension in consecutive degrees, and torsion freeness of the corresponding
	higher direct image sheaves. 
	
	\begin{theorem}[= Theorem~\ref{Main}]\label{intro-main-criterion}
		Let $f:X\to Y$ be a proper flat morphism of complex spaces to a smooth Stein
		curve and $\mathscr F$ a coherent analytic sheaf on $X$ which is flat with respect to $f$. Fix $q\geq0$. Then the following conditions are equivalent:
		\begin{itemize}
			\item[(i)] $\mathscr F$ is cohomologically flat in dimension $q$ over
			$Y$.
			\item[(ii)] The base change maps in degrees $q$ and $q-1$,
			\[
			\varphi^j(y):(R^jf_*\mathscr F)_y
			\otimes_{\mathcal O_{Y,y}}\mathbb C(y)
			\longrightarrow H^j(X_y,\mathscr F_{(y)}),
			\qquad j=q,q-1,
			\]
			are isomorphisms for every $y\in Y$.
			\item[(iii)] The function
			\[
			y\longmapsto\dim H^q(X_y,\mathscr F_{(y)})
			\]
			is locally constant.
			\item[(iv)] The restriction maps
			\[
			r^j:H^j(X,\mathscr F)\longrightarrow
			H^j(X_y,\mathscr F_{(y)}),
			\qquad j=q,q-1,
			\]
			are surjective for every $y\in Y$.
			\item[(v)] The sheaves $R^qf_*\mathscr F$ and
			$R^{q+1}f_*\mathscr F$ are torsion free.
			\item[(vi)] The jumping classes satisfy
			\[
			J^{q+1}=0 \quad\text{and}\quad J^q=0.
			\]
		\end{itemize}
	\end{theorem}
	
	Let us outline the proof of Theorem~\ref{intro-main-criterion}. Since the base is a
	smooth curve, its local rings are discrete valuation rings. The
	cohomological flatness theorem
	\cite[Theorem~4.12.(ii), Chapter~III]{BS76} relates cohomological flatness to
	the local constancy of fiberwise cohomology dimensions. The
	universal coefficient theorem and base change exact sequences identify the
	failures of base change and of extension in consecutive degrees with torsion in the
	corresponding higher direct images. Finally, the identities of
	Griffiths~\cite[Theorems~1.1 and~4.2]{Gr65} express cohomology jumps in terms of the
	jumping classes. Combining these ingredients, we obtain the equivalence of
	the six conditions in Theorem~\ref{intro-main-criterion}.
	Following Rao--Tsai~\cite{RT1}, we use exact
	sequences associated with a local parameter to make the obstruction caused
	by torsion explicit in Appendix \ref{app-base change}. 
	
	Following A.~Teleman~\cite{Teleman15}, we relate
	torsion in higher direct images to
	Brill--Noether loci and exclude it by codimension estimates based on
	Cohen--Macaulay rings and regular sequences. This leads to the following
	criterion.
	\begin{theorem}[= Theorem \ref{cri}]\label{intro-brill-noether}
		Let $B$ and $M$ be complex analytic spaces, $\pi:M\to B$ a flat morphism, and $\mathcal E$ a holomorphic vector bundle on $M$.
		Suppose that $B$ is Cohen--Macaulay and that for every $\max\{k-\dim B,0\}\leq i\leq k-1$,
		$R^{i}\pi_*(\mathcal{E})=0$ and
		$\operatorname{codim}_{B}(\BN_\pi^i)\geq k-i+1$ for the {i-th Brill--Noether locus} $\BN_\pi^i$. Then
		$R^k\pi_*(\mathcal{E})$ is torsion free.
	\end{theorem}
	
	We next explain the argument for Theorem~\ref{intro-brill-noether}. The
	codimension hypotheses on the Brill--Noether
	loci, together with the Cohen--Macaulay property of the base, provide the
	regular sequences needed to exclude torsion in $R^k\pi_*\mathcal E$. Thus
	the geometry of the jumping loci is converted into a torsion freeness
	criterion for the higher direct image.
	
	With these general cohomological tools in place, we turn to applications to
	pluricanonical systems. Our first application concerns the deformation
	invariance of plurigenera for K\"ahler threefolds with canonical singularities.
	
	\begin{theorem}[= Theorem~\ref{t-threefold-plurigenera}]\label{intro-threefold-plurigenera}
		Let $f:X\to S$ be a $\mathbb{Q}$-factorial K\"ahler family of compact threefolds over a
		relatively compact, smooth, connected curve. Assume that $X_0$ has
		canonical singularities. Then, for every integer $m\geq1$, the $m$-th
		plurigenus $P_m(X_t)$
		is independent of $t$ near $0$.
	\end{theorem}
	
	To prove Theorem~\ref{intro-threefold-plurigenera}, we run a relative
	minimal model program over $S$, following Nakayama~\cite{Naka85} and
	\cite{DHP23,DO24,DO26}. If some step contracts the central fiber, then
	$P_m(X_0)=0$ for every $m\geq1$. By upper semicontinuity,
	$P_m(X_t)=0$ for all $t$ sufficiently close to $0$; hence $P_m$ is $0$ and locally
	constant. Otherwise, the minimal model program terminates
	with either a relative minimal model or a Mori fiber space. In the first
	case, semiampleness and Koll\'ar's torsion freeness (Proposition~\ref{p-kollar-torsion freeness}) give base change for
	pluricanonical sections and hence invariance of plurigenera. In the latter case (i.e. the Mori fiber space case), one can show that
	$P_m(X_t)=0$ for every $m\geq1$ and all $t$ sufficiently close to $0$.
	
	A second application concerns smooth K\"ahler families of varieties with maximal Albanese
	dimension. We generalize J.~A.~Chen--Hacon's deformation-invariance result for
	Kodaira dimension~\cite[Corollary~2]{CH01} by proving deformation invariance
	of plurigenera in this setting. The proof uses results of O.~Das--Hacon~\cite{DHAlbanese} and
	O.~Fujino~\cite{Fujinoalb} on varieties of maximal Albanese dimension.
	
	\begin{theorem}[= Theorem~\ref{t-maxAlbdim}]\label{intro-maxAlbdim}
		Let $f:X\to S$ be a smooth K\"ahler family over a relatively compact,
		smooth, connected curve, and assume that $X_0$ has maximal Albanese
		dimension. Then, for every $m\geq1$, the $m$-th plurigenus $P_m(X_t)$ is independent of $t$
		near $0$.
	\end{theorem}
	
	To prove Theorem~\ref{intro-maxAlbdim}, we first show that having maximal Albanese
	dimension is stable under deformation. The relative Albanese map is
	therefore generically finite, so $K_X$ is big over the relative Albanese
	variety. The K\"ahler minimal model program then produces a good minimal model
	over this relative Albanese variety. Since the fibers of the relative Albanese variety are complex
	tori and contain no rational curves, Lemma~\ref{l-refnefalb} shows that the canonical divisor of the output of the mmp is nef over $S$. The theorem of
	Das--Hacon~\cite{DHAlbanese}, together with the relative base point free theorem, then
	gives a relative good minimal model over $S$. Finally,
	Proposition~\ref{p-kollar-torsion freeness} yields the constancy of the plurigenera.
	
	We next shift our attention to log plurigenera. Using the extension theorem for the minimal
	model program developed in \cite[Proposition~4.16]{HLR26} (see also
	Theorem~\ref{ExtMMP}), we obtain a log analogue of Koll\'ar's
	theorem~\cite[Theorem 1]{DefGeneralType}.
	
	\begin{theorem}[= Theorem~\ref{t-log-plurigenera}]\label{intro-log-plurigenera}
		Let $g:X\to S$ be a smooth family such that $(X,B)$ is log smooth over a relatively compact, smooth, connected curve $S$. Suppose that the special fiber $(X_0,B_0)$ is projective with canonical singularities, that $\lfloor B\rfloor=0$, and that $K_{X_0}+B_0$ is big.
		Assume moreover that the negative part of its Boucksom--Zariski
		decomposition satisfies
		$$
		N(K_{X_0}+B_0)\wedge B_0 = 0 .
		$$
		Then there exists an open neighborhood $U\subset S$ of $0$ such that,
		for every integer $m\geq1$ such that $mB$ is integral, the $m$-th log plurigenus
		$P_m(X_t,B_t)$
		is independent of $t\in U$.
	\end{theorem}
	
	The proof of Theorem~\ref{intro-log-plurigenera} begins with the log minimal
	model program on the projective central fiber, and
	Theorem~\ref{ExtMMP} extends its steps over a neighborhood of
	$0$. Kawamata--Viehweg vanishing, together with Theorem~\ref{Main}, then
	gives the local constancy of the fiberwise cohomology dimensions for $P_m(X_t,B_t)$ when
	$m\geq2$. The case $m=1$ is treated by \cite[Complement 2.62.5]{Kolbook}.
	
	For higher cohomological plurigenera, the extension theorems of
	\cite{CDM17,ZZ22} provide the condition required by
	Theorem~\ref{Main} and yield the following result.
	
	\begin{theorem}[= Corollary~\ref{c-higher-plurigenera}]\label{intro-higher-plurigenera}
		Let $f:X\to S$ be a smooth family of compact K\"ahler manifolds over a
		relatively compact, smooth, connected curve. If $K_X$ is \emph{semipositive} (i.e., admits a smooth Hermitian metric with semipositive Chern curvature),
		then, for every $m\geq1$ and $q\geq0$, the higher cohomological plurigenus $P_m^{(q)}(X_t)$
		is constant on $S$.
	\end{theorem}

	Finally, we extend Li--Wang's result
	\cite[Corollary~1.3]{LW26} from projective families to Moishezon families.
	
	\begin{theorem}[= Theorem~\ref{t-ginvarpluri}]\label{intro-ginvarpluri}
		Let $f:X\to S$ be a proper Moishezon morphism with connected fibers over
		a relatively compact, smooth, connected curve and
		$(X,\mathbf M)$ a generalized pair with abundant nef part over $S$. Write
		\[
		K_{X_t}+\mathbf M_t
		=\left.(K_X+\mathbf M_X)\right|_{X_t},
		\]
		and assume that $(X_t,\mathbf M_t)$ has generalized canonical
		singularities for every $t\in S$. Then, for every $m\geq1$ such that
		$m\mathbf M_X$ is Cartier, the generalized $m$-th plurigenera
		\[
		t\longmapsto
		\dim H^0\!\left(X_t,
		\mathcal O_{X_t}\bigl(m(K_{X_t}+\mathbf M_t)\bigr)\right)
		\]
		is constant on $S$.
	\end{theorem}
	
	The proof of Theorem~\ref{intro-ginvarpluri} is largely inspired by the
	approach of Li--Wang~\cite{LW26}, particularly their construction of singular
	metrics for abundant nef parts and their extension theorem for generalized
	pluricanonical sections. Our main new input is that, locally over each point
	of $S$, the abundance nef part hypothesis yields a factorization through a projective
	morphism $h:Z\to U$ for some neighborhood $U\subset S$, although the original
	family is only Moishezon (see
	Proposition~\ref{p-nefabundant}). This auxiliary projective model allows us to
	carry out Li--Wang's metric construction in the Moishezon setting, as in
	Lemma~\ref{l-bignef}. The resulting curvature and integrability properties
	give the extension theorem for generalized pluricanonical sections (see
	Proposition~\ref{t-LWext}), which completes the proof.

	\textbf{Notation and conventions.} Unless otherwise stated, all spaces and
	morphisms are complex analytic. For a morphism $f:X\to S$ and a point
	$s\in S$, we denote the fiber by $X_s:=f^{-1}(s)$. For a coherent sheaf
	$\mathscr F$ on $X$, we write
	$\mathscr F_{(s)}:=\mathscr F\otimes_{\mathcal O_X}\mathcal O_{X_s}$.
	Thus $\mathscr F_{(s)}$ denotes a fiber restriction, whereas $\mathscr F_x$
	denotes the stalk at $x\in X$. We set
	$h^q(X,\mathscr F):=\dim_{\mathbb C}H^q(X,\mathscr F)$. For two functions
	$\varphi$ and $\psi$, we write $\varphi\approx\psi$ if
	$\varphi-\psi$ is locally bounded. Line bundles are
	written additively: 
	$L+M:=L\otimes M, mL:=L^{\otimes m}, \text{and}
	-L:=L^\vee.$
	Finally, if $D=\sum d_iD_i$ and $E=\sum e_iD_i$ are
	$\mathbb R$-divisors written with respect to a common collection of prime
	divisors (allowing zero coefficients), we define their \emph{coefficientwise minimum} by $D\wedge E:=\sum_i\min\{d_i,e_i\}D_i$. A family is said to be \emph{K\"ahler} if its total space is K\"ahler, while a family is \emph{Moishezon} if it is bimeromorphic to a projective family. 
	
	\textbf{Organization of the paper.} Section~\ref{sec-preliminaries}
	collects preliminaries on torsion, base change, singularities, and MMP steps
	in families. Section~\ref{sec-cohomological-flatness} establishes criteria for deformation invariance of cohomological dimension, and a torsion-freeness criterion in terms of higher Brill--Noether loci. Section~\ref{sec-invariance} proves our main
	results on plurigenera, log plurigenera, higher cohomological plurigenera, and generalized plurigenera.
	Appendix~\ref{app-base change} gives Rao--Tsai's dimension-invariance approach to
	torsion freeness.
	
	\textbf{Acknowledgments.} The last two authors are grateful to Professor C.~D.~Hacon for his collaboration on
	\cite{HLR26}; the present paper may be viewed, in part, as a continuation of that work. The authors also thank Professor Zhan Li for many helpful discussions on deformation invariance of plurigenera.
	
	\section{Preliminaries}\label{sec-preliminaries}
	
	This section collects background on torsion sheaves, higher direct images,
	log pairs, and generalized pairs that will be used in the sequel. We conclude
	with extension results for contractions and for steps of the MMP.

	\subsection{Torsion in groups and sheaves}\label{notation-con}
	In this subsection, we introduce definitions and properties of torsion groups and torsion sheaves. Throughout this subsection, $R$ denotes a commutative ring, and $M, N, P$ denote $R$-modules.
	\begin{definition}
		An $R$-module $P$ is called \emph{projective} if whenever $p$ is surjective and $h$ is any map, there exists a map $g$ such that the following diagram is commutative:
		\[\begin{tikzcd}
			& P & \\
			M & {N} & 0
			\arrow["g"',dashed,from=1-2, to=2-1]
			\arrow["h", from=1-2, to=2-2]
			\arrow["p"', two heads, from=2-1, to=2-2]
			\arrow[from=2-2, to=2-3].
		\end{tikzcd}\]
	\end{definition}
	\begin{definition}
		Let $T:\mathcal{A} \to \mathcal{C}$ be an additive covariant functor between abelian categories, and suppose that $\mathcal{A}$ has enough projectives. For every object $A$, choose a projective resolution
		$$
		\textbf{P}=\cdots \to P_2 \to P_1 \to P_0 \to A \to 0.
		$$
		Taking homology, we define the \emph{left derived functors} by
		$$
		(L_nT)(A)=H_n(T\textbf{P}_A),
		$$
		where $\textbf{P}_A=\cdots\to P_2\to P_1\to P_0\to0$ denotes the
		associated chain complex.
	\end{definition}
	\begin{definition}
		Recall that the functor $T(M)=M\otimes_R N$ is a right exact functor from $\mathbf{mod}$-$R$ to $\mathbf{Ab}$.
		Then its left derived functors are called \emph{torsion groups}:
		$$
		\Tor_n^R(M,N)=(L_nT)(M).
		$$
	\end{definition}
	Now we state the relation between a torsion free module and its torsion group through the projective dimension below.\par
	Let $R$ be a Noetherian local ring with residue field $k$ and $M$ a finitely generated $R$-module.
	\begin{definition}[{\cite[Definition~4.1.1]{Weibel94}}]The \emph{projective dimension} of $M$, denoted by $\operatorname{pd}(M)$, is the minimum integer $n$ (if it exists) such that there is a resolution of $M$ by projective modules:
		$$
		0 \to P_n \to \cdots \to P_1 \to P_0 \to M \to 0.
		$$
	\end{definition}
	\begin{proposition}
		For any integer $d \geq 0$, we have
		\begin{itemize}
			\item[(i)]{\cite[Proposition~3.1]{Mat89}}\label{Prop 2.5}
			The $R$-module $M$ is free if and only if $M$ is projective.
			\item[(ii)]{\cite[Proposition~4.4.11]{Weibel94}}\label{pd-Tor-criterion}
			The projective dimension $\operatorname{pd}(M)\leq d$ if and only if $\Tor^R_{d+1}(M,k)=0$.  
		\end{itemize}
	\end{proposition}
	\begin{corollary}\label{2.7}
		The torsion group $\Tor^R_{1}(M,k)=0$ if and only if $M$ is a free $R$-module.
	\end{corollary}
	\begin{proof}
		If $\Tor^R_{1}(M,k)=0$, then $\operatorname{pd}(M)=0$. So $M$ has a projective
		resolution of length zero. Hence $M$ is projective and therefore free by
		Proposition~\ref{Prop 2.5}. Conversely, if $M$ is free, then it is
		projective, so $\operatorname{pd}(M)=0$ and $\Tor^R_1(M,k)=0$ by
		Proposition~\ref{pd-Tor-criterion}.
	\end{proof}
	Furthermore, over an integral domain $R$, every finitely generated free
	$R$-module is torsion free. Thus, under this assumption,
	$\Tor^R_{1}(M,k)=0$ implies that $M$ is torsion free. For the converse, we should further assume that $R$ is a principal ideal domain, as in the structure
	theorem below.
	\begin{proposition}[{\cite[\S~12.1, Theorem~5]{Dum04}}]\label{Torzero}
		Let $R$ be a principal ideal domain. Then
		\begin{itemize}
			\item[(i)] $M$ is isomorphic to the direct sum of finitely many cyclic modules. More precisely,
			$$
			M \cong R^r \oplus R/(a_1) \oplus R/(a_2) \oplus \cdots \oplus R/(a_m)
			$$
			for some integer $r \geq 0$ and nonzero elements $a_1, a_2, \cdots, a_m$ of $R$ which satisfy the divisibility relations
			$$
			a_1 |a_2| \cdots |a_m.
			$$
			\item[(ii)] $M$ is torsion free if and only if $M$ is free.
			\item[(iii)] In the decomposition in (i),
			$$
			\Tor(M) \cong R/(a_1) \oplus R/(a_2) \oplus \cdots \oplus R/(a_m),
			$$
			where $\Tor(M):=\{x \in M\ |\ rx=0$ for some nonzero $r \in R\}$ is the \emph{torsion submodule} of $M$. 
		\end{itemize}
	\end{proposition}
	When $R$ is a principal ideal domain, we also obtain the following universal coefficient theorem.
	\begin{theorem}[{Universal coefficient theorem, \cite[Theorem~5.10]{Spanier95}}]\label{UCT}
		For any cochain complex of free $R$-modules $C^\bullet$, there is a split exact sequence
		$$
		0 \to H^q(C^\bullet) \otimes_R M \to H^q(C^\bullet \otimes_R M) \to \Tor_1^R(H^{q+1}(C^\bullet),M) \to 0.
		$$
	\end{theorem}
	A sheaf-theoretic analogue of Proposition \ref{Torzero} will be presented in Proposition~\ref{jump-dimension}. As a first step toward this geometric setting, we now pass from torsion in modules to the corresponding notion for sheaves.
	\begin{definition}
		Let $(X,\mathcal{O}_X)$ be a reduced complex space. For every $\mathcal{O}_X$-module $\mathscr{S}$ we define 
		$$
		T(\mathscr{S}):= \bigcup_{x \in X} T(\mathscr{S}_x), 
		$$
		where $T(\mathscr{S}_x):=\{s_x \in \mathscr{S}_x: g_xs_x=0$ for a suitable nonzero divisor $g_x \in \mathcal{O}_{X,x}$\}. 
		We call $T(\mathscr{S})$ the \emph{torsion sheaf} of $\mathscr{S}$, and we
		say that $\mathscr{S}$ is \emph{torsion free} at $x$ if
		$T(\mathscr{S})_x=0$.
	\end{definition}
	\begin{lemma}[{\cite[\S~7.5]{Remmert94}}]\label{torsion-sheaf}
		Let $\mathscr{S}$ be a coherent analytic sheaf on $X$. Then $T(\mathscr{S})$ is the kernel of the natural $\mathcal{O}$-homomorphism $\sigma: \mathscr{S} \to \mathscr{S}^{**}$ of $\mathscr{S}$ into its bidual. In particular, $T(\mathscr{S})$ is coherent.
	\end{lemma}
	\begin{lemma}
		For any coherent analytic sheaf $\mathscr{F}$ on a curve $X$, if $\mathscr{F}$ is torsion free at a point $x$, then it is torsion free over the dense Zariski open neighborhood of $x$.	    
	\end{lemma}
	
	\begin{proof}
		By Lemma \ref{torsion-sheaf}, 
		$T(\mathscr{F})$ is a coherent subsheaf of $\mathscr{F}$. Therefore, its
		support $\Supp(T(\mathscr{F}))$ is a closed analytic subset of $X$. Since
		$X$ is a smooth curve, a proper closed analytic subset is discrete. Because
		$(T(\mathscr{F}))_x=0$, the support is proper and does not contain $x$.
		Thus $U=X\setminus\Supp(T(\mathscr{F}))$ is a Zariski open dense
		neighborhood of $x$, and $\mathscr{F}$ is torsion free on $U$.
	\end{proof}
	\begin{proposition}\label{p-pushforward-torsion free}
		Let $X,Y$ be normal complex analytic varieties, $f:X\to Y$ a dominant holomorphic map,
		and $\mathscr{F}$ a torsion free coherent analytic sheaf on $X$. Then
		$f_*\mathscr{F}$ is torsion free.
	\end{proposition}
	\begin{proof}
		Fix $y\in Y$. Suppose that $t\in(f_*\mathscr{F})_y$ and that there is a
		nonzero holomorphic germ $u\in\mathcal{O}_{Y,y}$ such that $u\cdot t=0$.
		We show that $t=0$. Choose a connected open neighborhood $V \subset Y$
		and a section $a\in \mathcal{O}_{Y}(V)$ whose germ at $y$ is $u$, such
		that $t$ is represented by a section
		\[
		s \in f_*\mathscr{F}(V)=\mathscr{F}(f^{-1}(V)).
		\]
		Then $u\cdot t=0$ means $(f^*a)\cdot s=0$ on $f^{-1}(V)$.
		Since $a$ is nonzero and $V$ is connected, the zero set $Z(a)$ is a proper analytic subset of $V$; in particular,
		\[
		\dim Z(a)<\dim Y.
		\]
		Now suppose that there exists a nonempty open subset $U \subset X$ such that $f(U) \subset Z(a)$. Then $f^*a=a \circ f=0$ on $U$. By the identity theorem, $a \circ f=0$ on $X$, i.e., $f(X) \subset Z(a)$, contradicting the dominance of $f$. Thus $f^*a=a\circ f \neq 0$ on every open subset of $f^{-1}(V)$. Since $X$ is normal, $\mathcal{O}_{X,x}$ is an integral domain for every $x \in X$. Hence $(f^*a)_x$ cannot be a zero divisor.
		If $(f^*a)_x\cdot s_x=0$, then $s_x=0$ for every $x\in f^{-1}(V)$, so $s=0$. Hence $t=0 \in (f_*\mathscr{F})_y$, which completes the proof.
	\end{proof}
	\subsection{Higher direct images and base change}
	In this subsection, we refer to \cite{BS76} for properties of higher direct images.
	\begin{definition}[{\cite[p. 92]{BS76}}]
		Let $f:X \to Y$ be a continuous map of topological spaces and $\mathscr{F}$ a sheaf of abelian groups on $X$. Denote by $R^qf_*(\mathscr{F})$ the sheaf associated to the presheaf on $Y$
		$$
		V \longmapsto H^q(f^{-1}(V),\mathscr{F}),
		$$
		whose restrictions are naturally defined for any $q \geq 0$. The sheaves $R^qf_*(\mathscr{F})$ are called the \emph{higher direct images} of $\mathscr{F}$.
	\end{definition}
	In fact, $R^\bullet f_*$ are just the right derived functors of the functor $f_*:\textbf{Ab}(X) \to \textbf{Ab}(Y)$.
	\begin{proposition}[{\cite[Theorem~6.2]{Ive86}}]\label{higher stalk}
		Let $f: X \to Y$ be a proper map between locally compact spaces and $\mathscr{F}$ a sheaf of abelian groups on $X$. For any point $y \in Y$ and integer $q \geq 0$,
		$$
		(R^qf_*(\mathscr{F}))_y \simeq H^q(f^{-1}(y),\mathscr{F}).
		$$
	\end{proposition}
	\begin{theorem}[{Grauert direct image theorem,
			\cite[Theorem~2.1, Chapter~III]{BS76}}]\label{direct image}
		Let $f:X \to Y$ be a proper morphism of complex spaces and $\mathscr{F}$ a coherent analytic sheaf on $X$. Then the analytic sheaves $R^qf_*(\mathscr{F})$ are coherent for any $q \geqslant 0$.
	\end{theorem}
	\begin{definition}[{\cite[p. 116]{BS76}}]
		Let $f: X \to Y$ be a morphism of complex spaces and $y \in Y$. We denote by $\mathfrak{m}_y$ both the maximal ideal of $\mathcal{O}_{Y,y}$ and the natural ideal sheaf defined by it; by $\hat{\mathfrak{m}}_y$ we mean the ideal sheaf of $\mathcal{O}_X$ generated by the inverse image of $\mathfrak{m}_y$. For any coherent analytic sheaf $\mathscr{F}$ on $X$, the quotient $$\mathscr{F}_{(y)}:= \mathscr{F}/\hat{\mathfrak{m}}_y \mathscr{F}$$ is called the \emph{analytic fiber} of $\mathscr{F}$ over $y$.
	\end{definition}
	\begin{theorem}[{Grauert base change theorem,
			\cite[Theorem~3.4, Chapter~III]{BS76}}]\label{base change}
		Let $f:X \to Y$ be a proper morphism of complex spaces, $\mathscr{F}$ a coherent analytic sheaf on $X$ which is flat with respect to $f$, $y$ a point in $Y$ and $q$ an integer. The following assertions are equivalent:
		\begin{itemize}
			\item[(a)] For an arbitrary $\mathcal{O}_y$-module $M$ of finite type, the canonical morphism 
			$$R^qf_*(\mathscr{F})_y \otimes_{\mathcal{O}_{Y,y}} M \to R^qf_*(\mathscr{F} \otimes_{\mathcal{O}_X}f^*M)_y
			$$ 
			is an isomorphism.
			\item[(b)] The canonical morphism 
			$$
			R^qf_*(\mathscr{F})_y \to R^qf_*(\mathscr{F}/\hat{\mathfrak{m}}_y\mathscr{F})_y
			$$
			is surjective.
		\end{itemize}
	\end{theorem}
	\begin{corollary}\label{iso-sur-lem}
		In the setting of Theorem \ref{base change}, if the base change map
		$$
		\varphi^q(y):(R^qf_*\mathscr{F})_y\otimes_{\mathcal{O}_{Y,y}}\mathbb{C}(y) \to H^q(X_y,\mathscr{F}_{(y)})
		$$
		is surjective, then it is an isomorphism.
	\end{corollary}
	\begin{proof}
		The canonical map in Theorem \ref{base change}.(b) factors as the natural map followed by the base change map:
		$$
		R^qf_*(\mathscr{F})_y \twoheadrightarrow R^qf_*(\mathscr{F})_y \otimes \mathbb{C}(y) \xrightarrow{\varphi^q(y)} R^qf_*(\mathscr{F}/\hat{\mathfrak{m}}_y\mathscr{F})_y \cong H^q(X_y,\mathscr{F}_{(y)}).
		$$
		Since the natural map is surjective, $\varphi^q(y)$ is surjective if and only if the canonical map $R^qf_*(\mathscr{F})_y \to R^qf_*(\mathscr{F}/\hat{\mathfrak{m}}_y\mathscr{F})_y$ is surjective. Therefore by Theorem \ref{base change}, if the base change map $\varphi^q(y)$ is surjective, then it is an isomorphism.
	\end{proof}
	\subsection{Commutative algebra preliminaries}
	We collect the notions from commutative algebra needed later to translate codimension estimates into vanishing statements. In particular, regular sequences and the equality between grade and height in Cohen--Macaulay rings will be used in Subsection~\ref{app-brill-noether} to control torsion in higher direct image sheaves.
	\begin{definition}[{\cite[p. 6]{Hart77}}]
		For a ring $A$, the \emph{height} of a prime ideal $\mathfrak p$, denoted by $\operatorname{ht}(\mathfrak p)$, is the supremum of all integers $n$ such that there exists a chain $\mathfrak p_0 \subset \mathfrak p_1 \subset \cdots \subset \mathfrak p_n=\mathfrak p$ of distinct prime ideals.
	\end{definition}
	\begin{definition}[{\cite[15.A]{Mat89}}]
		Let $R$ be a ring, $M$ an $R$-module and $x_1, x_2, \cdots, x_r$ a sequence of elements of $R$. We say $x_1, x_2, \cdots, x_r$ is an \emph{$M$-regular sequence} if the following conditions are satisfied:
		\begin{itemize}
			\item[(a)] for each $1 \leq i \leq r$, $x_i$ is not a zero-divisor on $M/(x_1, x_2, \cdots, x_{i-1})M$ and
			\item[(b)] $M \neq (x_1, x_2, \cdots, x_r)M$.
		\end{itemize}
		When all $x_i$ belong to an ideal $I$, we say $x_1, x_2, \cdots, x_r$ is an \emph{$M$-regular sequence in $I$}. If, moreover, there is no $b \in I$ such that $x_1, x_2, \cdots, x_r,b$ is \emph{$M$-regular}, then $x_1, x_2, \cdots, x_r$ is said to be a \emph{maximal $M$-regular sequence in $I$}.
	\end{definition}
	\begin{remark}
		For a local ring $R$ with maximal ideal $\mathfrak{m}$, if $M \neq 0$ is a finitely generated $R$-module and $(x_1, x_2, \cdots, x_r) \subset \mathfrak{m}$, condition $(b)$ is satisfied automatically by Nakayama's lemma.
	\end{remark}
	\begin{definition}[{\cite[Definition~1.2.7]{BH93}}]
		Let $(R,\mathfrak{m},k)$ be a Noetherian local ring and $I$ a proper ideal of $R$. We call the common length of the maximal $R$-sequences in $I$ the \emph{grade} of $I$, denoted by $\operatorname{grade}(I)$. In particular, we call the grade of $\mathfrak{m}$ the \emph{depth} of $R$, denoted by depth$(R)$. We say that $R$ is a \emph{Cohen--Macaulay ring} if depth$(R)=\dim R$.
	\end{definition}
	\begin{lemma}[{\cite[Corollary~2.1.4]{BH93}}]\label{CM}
		If $R$ is a Cohen--Macaulay ring and $I$ is a proper ideal of $R$, then
		$\operatorname{grade}(I)=\operatorname{ht}(I)$ and
		$\operatorname{ht}(I)+\dim(R/I)=\dim R$.
	\end{lemma}

	\subsection{Log pairs, generalized pairs, and singularities}
	We next shift our attention to the basic notions of log pairs, generalized pairs and singularities in the complex analytic setting. 
	\begin{definition}[Log pairs]
		A \emph{log pair} $(X,B)$ consists of a normal complex analytic variety $X$ and an effective $\mathbb{R}$-Weil divisor $B$ such that $K_X+B$ is $\mathbb{R}$-Cartier. Here $K_X$ denotes the canonical sheaf $\omega_X$ in divisor notation. If the coefficients of $B$ lie in $[0,1]$, then $B$ is called a \emph{boundary}. 
	\end{definition}
	Given a log pair $(X,B)$, write $B=\sum_j b_jB_j$ with $B_j$ prime, we say that it is \emph{log smooth} over the base
	$U$ if $X\to S$ is smooth and $B_I\to S$ is smooth for every stratum
	\[
	B_I=\bigcap_{j\in I}B_j=B_{i_1}\cap\cdots\cap B_{i_r},
	\qquad I=\{i_1,\ldots,i_r\}.
	\]
	
	\begin{definition}[Generalized pairs, {\cite[Definition~2.1]{HLR26}}]
		Let $f:X\to S$ be a proper surjective morphism of complex analytic varieties, with $X$ normal and $S$ relatively compact. An \emph{analytic generalized pair} $(X/S,B+\bbeta)$ consists of an effective $\mathbb{R}$-Weil divisor $B$ and a positive closed real $b$-$(1,1)$-current $\bbeta$ for which there exist a resolution $\nu:X'\to X$ and an $\mathbb{R}$-divisor $B'$ on $X'$ satisfying:
		\begin{enumerate}[(i)]
			\item $B'$ has simple normal crossing support and $\nu_*B'=B$;
			\item $\bbeta$ descends to $X'$, and its trace $\bbeta_{X'}$ has local potentials with $[\bbeta_{X'}]\in H^{1,1}_{\mathrm{BC}}(X',\mathbb{R})$ nef over $S$;
			\item there is a class $\gamma\in H^{1,1}_{\mathrm{BC}}(X,\mathbb{R})$ such that
			\[
			[K_{X'}+B'+\bbeta_{X'}]=\nu^*\gamma.
			\]
		\end{enumerate}
		Here a \emph{positive closed real $b$-$(1,1)$-current} is a collection of positive closed real $(1,1)$-currents $\bbeta_Y$ on the proper bimeromorphic models $Y\to X$, compatible with pushforward: $p_*\bbeta_{Y_1}=\bbeta_{Y_2}$ for every morphism $p:Y_1\to Y_2$ over $X$. The current $\bbeta_Y$ is called its \emph{trace} on $Y$. \emph{Descent} to $X'$ means that $\bbeta_{X'}$ has local potentials and
		\[
		\bbeta_Y=\mu^*\bbeta_{X'}
		\]
		for every higher model $\mu:Y\to X'$. We call $\gamma$ the \emph{adjoint class} and denote it by $[K_X+B+\bbeta_X]$. If $X$ is K\"ahler, the pair $(X/S,B+\bbeta_X)$ is called a \emph{generalized K\"ahler pair}.
		
		In the $\mathbb Q$-divisorial setting, we replace $\bbeta$ by a
		$b$-$\mathbb Q$-Cartier $b$-divisor $\mathbf M$ and write
		$(X/S,B+\mathbf M)$. More precisely, a \emph{$\mathbb Q$-divisorial
			generalized pair} consists of an effective $\mathbb Q$-Weil divisor $B$
		and a $b$-divisor $\mathbf M$ descending to a resolution $\nu:X'\to X$,
		whose trace $\mathbf M_{X'}$ is $\mathbb Q$-Cartier and nef over $S$,
		such that $K_X+B+\mathbf M_X$ is $\mathbb Q$-Cartier. Descent means that
		$\mathbf M_Y=\mu^*\mathbf M_{X'}$ for every higher model $\mu:Y\to X'$.
		The divisor $B'$ is determined by the equality
		\[
		K_{X'}+B'+\mathbf M_{X'}=\nu^*(K_X+B+\mathbf M_X),
		\]
		and we choose $X'$ so that $B'$ has simple normal crossing support.
		
	\end{definition}
	
	Following \cite[Definition~2.2]{HLR26}, let $(X,B+\bbeta)$ be an analytic generalized pair, with defining data $\nu:X'\to X$ and $B'$ as above. For a prime divisor $E$ over $X$, choose a smooth higher model $\mu:Y\to X'$ on which $E$ appears, and define $B_Y$ by
	\[
	K_Y+B_Y=\mu^*(K_{X'}+B').
	\]
	Since $\bbeta_Y=\mu^*\bbeta_{X'}$, this is the crepant pullback of the generalized pair. The \emph{discrepancy} of $E$ is
	\[
	a(E,X,B+\bbeta):=-\operatorname{coeff}_E(B_Y),
	\]
	independent of the chosen higher model. The corresponding \emph{log discrepancy} is $1+a(E,X,B+\bbeta)$.
	For an ordinary log pair $(X,B)$, one instead uses $K_Y+B_Y=\pi^*(K_X+B)$ on a log resolution $\pi:Y\to X$ and writes $a(E,X,B)$.
	
	We say that $(X,B+\bbeta)$ is \emph{generalized klt} if $a(E,X,B+\bbeta)>-1$ for every prime divisor $E$ over $X$, and that it has \emph{canonical singularities} if $a(E,X,B+\bbeta)\geq 0$ for every exceptional divisor $E$ over $X$.
	For ordinary log pairs, the same conditions define \emph{klt singularities} and \emph{canonical singularities} using $a(E,X,B)$. A normal variety $X$ is \emph{klt} or \emph{canonical} if $(X,0)$ is so. Note that the klt condition also tests divisors on $X$, whereas the canonical condition tests only exceptional divisors.

	\subsection{Extending contractions and MMP steps}
	We first recall a result of Koll\'ar and Mori~\cite[Proposition~(11.4)]{KM92},
	which gives conditions under which a contraction morphism extends under a small
	deformation.
	
	\begin{theorem}[Extension of a contraction morphism]\label{ExtendCont}
		Let $X\to S$ be a proper flat family over a smooth, connected, relatively compact
		curve and $\phi_0:X_0\to X_0'$ a proper surjective morphism of normal
		compact complex analytic spaces. Assume that
		$(\phi_0)_*\mathcal{O}_{X_0}=\mathcal{O}_{X_0'}$ and
		$R^1(\phi_0)_*\mathcal{O}_{X_0}=0$. If $f:X\to S$ is a small deformation
		of $X_0$, then $\phi_0$ can be extended to a contraction morphism
		$\phi:X\to X'$ over $S$:
		\[\begin{tikzcd}
			{X_0} && X \\
			& {X'_0} & {} & {X'} \\
			{\{0\}} && {S}
			\arrow[hook, from=1-1, to=1-3]
			\arrow["{\phi_0}", from=1-1, to=2-2]
			\arrow[from=1-1, to=3-1]
			\arrow["\phi", from=1-3, to=2-4]
			\arrow[from=1-3, to=3-3]
			\arrow[hook, no head, from=2-2, to=2-3]
			\arrow[from=2-2, to=3-1]
			\arrow[from=2-3, to=2-4]
			\arrow[from=2-4, to=3-3]
			\arrow[hook, from=3-1, to=3-3]. 
		\end{tikzcd}\]
	\end{theorem}
	\begin{proof}
		Although this result was proved in \cite[Proposition~(11.4)]{KM92}, we include a detailed proof. By Grauert's Kuranishi theorem \cite[Hauptsatz, p.~140]{Grauert74}, every compact complex space admits a versal deformation. We denote the corresponding spaces by $\text{Def}(X_0)$ and $\text{Def}(X_0')$, respectively.
		
		We first extend the contraction to the infinitesimal neighborhoods of $X_0$, and then to the formal neighborhood. Let $\mathcal{I}$ be the defining ideal sheaf of $X_0$ in $X$ and $nX_0$ the finite-order thickening $nX_0=(|X_0|,\mathcal{O}_X/\mathcal{I}^n)$. Since $S$ is a smooth curve and $f$ is flat, $X_0$ is Cartier; hence we have the exact sequence $$0\to \mathcal{O}_{X_0}(-jX_0)\to \mathcal{O}_{(j+1)X_0} \to \mathcal{O}_{jX_0}\to 0.$$
		After shrinking $S$ around $0$, the fiber $X_0$ is the principal Cartier divisor defined by the pullback of a local coordinate on $S$, so $X_0\sim 0$. Consequently,
		\[
		\mathcal{N}_{X_0/X}^{\vee}\cong\mathcal{O}_X(-X_0)|_{X_0}\cong\mathcal{O}_{X_0},
		\]
		and hence $\mathcal{O}_{X_0}(-jX_0)\cong\mathcal{O}_{X_0}$ for every $j\geq 1$. Thus
		\[
		R^1(\phi_0)_*\mathcal{O}_{X_0}(-jX_0)
		\cong R^1(\phi_0)_*\mathcal{O}_{X_0}=0.
		\]
		Applying $(\phi_0)_*$ to the preceding short exact sequence therefore yields
		\[
		0\to (\phi_0)_*\mathcal{O}_{X_0}(-jX_0)\to (\phi_0)_*\mathcal{O}_{(j+1)X_0}\to (\phi_0)_*\mathcal{O}_{jX_0}\to 0.\tag{*}
		\]
		Define $X_k' = \operatorname{Specan}((\phi_0)_*(\mathcal{O}_{kX_0}))$
		and let $\phi_k:kX_0\to X'_k$ be the induced morphism. Then
		$\phi_k$ restricts to the contraction $\phi_0$ on $X_0$, and $\phi_k$ are compatible when $k$ varies; thus $\phi_k$ is the
		$k$-th order infinitesimal thickening of $\phi_0$. The exact sequence $(*)$ gives a closed embedding $X_k' \hookrightarrow X_{k+1}'$, and these
		$X'_k$ for $k\ge 1$ are the infinitesimal thickenings of $X_0'$. Define a flat formal deformation $(X'_{\hat U},\mathcal{O}_{X'_{\hat U}})$ of $X'_0$ over the formal neighborhood $\hat U$ of $0$ in $S$, with $\mathcal{O}_{X'_{\hat U}}=\varprojlim_k\mathcal{O}_{X'_k}$.
		
		Next, we prove that the formal extension yields an analytic extension using the Douady space technique. Let $\pi': \mathcal{X}' \to \text{Def}(X_0')$ be a semi-universal deformation of $X_0'$, and consider $$f \times \pi': {X}\times \mathcal{X}'\to S \times \text{Def}(X_0').$$
		The relative Douady space $\mathscr{D} = \text{Douady}({X} \times \mathcal{X}' / S\times \text{Def}(X_0'))$ parametrizes compact complex subspaces of the fibers of $f\times\pi'$. Let $o=[\Gamma_{\phi_0}]\in\mathscr{D}$ be the point corresponding to the graph of $\phi_0:X_0\to X'_0$ and $\mathscr{D}_{\Gamma_{\phi_0}}$ an irreducible component containing $o$. Let $\Phi:\mathscr{D}_{\Gamma_{\phi_0}}\to S\times\text{Def}(X'_0)$ be the structure map of the relative Douady space, which sends $[\Gamma]$ to the pair $(s,t)$ such that $\Gamma\subset X_s\times\mathcal{X}'_t$. Let $\operatorname{pr}_1$ and $\operatorname{pr}_2$ be the projections from $S\times\text{Def}(X'_0)$ to $S$ and $\text{Def}(X'_0)$, respectively, and set $p_i:=\operatorname{pr}_i\circ\Phi$ for $i=1,2$, as shown in the diagram below.
		\begin{center}
			\begin{tikzcd}[ampersand replacement=\&]
				\&\& \mathscr{D}_{\Gamma_{\phi_0}}\\
				S \&\& {S\times \text{Def}(X'_0)} \&\& {\text{Def}(X'_0)}
				\arrow["{p_1}"', from=1-3, to=2-1]
				\arrow["\Phi"{description}, from=1-3, to=2-3]
				\arrow["{p_2}", from=1-3, to=2-5]
				\arrow[from=2-3, to=2-1]
				\arrow[from=2-3, to=2-5].
			\end{tikzcd}
		\end{center}
		
		For a point $[\Gamma]\in\mathscr{D}_{\Gamma_{\phi_0}}$, set
		$s=p_1([\Gamma])$ and $t=p_2([\Gamma])$. Thus $[\Gamma]$ represents a compact complex subspace $\Gamma\subset X_s\times\mathcal{X}'_t$. We define the \emph{graph locus} $Z\subset\mathscr{D}_{\Gamma_{\phi_0}}$ by
		\[
		Z:=\left\{[\Gamma]:
		\operatorname{pr}_1|_{\Gamma}:\Gamma\longrightarrow X_s
		\text{ is an isomorphism}\right\}.
		\]
		After shrinking $\mathscr{D}_{\Gamma_{\phi_0}}$ around $o$, the subspaces whose first projection is an isomorphism form an open graph locus $Z\subset\mathscr{D}_{\Gamma_{\phi_0}}$ (cf. \cite[Theorem I.1.10]{Kol96}).
		Since $\Gamma_{\phi_0}$ is the graph of $\phi_0$, the projection $\Gamma_{\phi_0}\to X_0$ is an isomorphism, and hence $o\in Z$. For each $k$, the graph of $\phi_k$ determines a morphism from the $k$th infinitesimal neighborhood of $0$ in $S$ to $\mathscr{D}$. The compatibility of the morphisms $\phi_k$ makes these maps compatible as $k$ varies, so they define a formal section $\hat\sigma:\hat U\to\mathscr{D}$ of the projection to $S$, with $\hat\sigma(0)=o$. Since $Z$ is open and contains $o$, this formal section factors through $Z$. Applying Artin's theorem~\cite[Theorem~1.5]{Artin1}, we obtain an open neighborhood $U\subseteq S$ of $0$ and a holomorphic section $\sigma:U\to Z$ of the projection $p_1|_Z: Z \to S$ with $\sigma(0) = o$. Let $h=p_2\circ\sigma:U\to\operatorname{Def}(X'_0)$ and let
		\[
		f':X'_U:=U\times_{\operatorname{Def}(X'_0)}\mathcal{X}'\longrightarrow U
		\]
		be the pullback of the semi-universal family $\pi':\mathcal{X}'\to\operatorname{Def}(X'_0)$ along $h$. Write $X_U=f^{-1}(U)$. For each $u\in U$, the point $\sigma(u)\in Z$ represents a graph $\Gamma_u\subset X_u\times\mathcal{X}'_{h(u)}$. Thence $\sigma$ is holomorphic, these graphs form a holomorphic family over $U$ whose total space is a subspace
		\[
		\Gamma_U\subset X_U\times_U X'_U.
		\]
		The fiber of $\Gamma_U$ over each $u\in U$ is $\Gamma_u$.
		Denote the two projections by $q_1$ and $q_2$, as in the following commutative diagram
		\begin{center}
			\begin{tikzcd}[ampersand replacement=\&]
				\& {\Gamma_U} \\
				{X_U} \& {X_U\times_U X'_U} \& {X'_U}
				\arrow["{q_1}"', from=1-2, to=2-1]
				\arrow[hook, from=1-2, to=2-2]
				\arrow["{q_2}", from=1-2, to=2-3]
				\arrow[from=2-2, to=2-1]
				\arrow[from=2-2, to=2-3].
			\end{tikzcd}
		\end{center}
		Since $\sigma$ takes values in the graph locus $Z$, the first projection $q_1:\Gamma_U\to X_U$ is an isomorphism. Thus the second projection $q_2:\Gamma_U\to X'_U$ defines a holomorphic morphism over $U$ as
		\[
		\phi:=q_2\circ q_1^{-1}:X_U\longrightarrow X'_U.
		\]
		For each $s\in U$, its graph is precisely the subspace $\Gamma_s$ represented by $\sigma(s)$. In particular, $\sigma(0)=o=[\Gamma_{\phi_0}]$ implies $\phi|_{X_0}=\phi_0$. Thus $\phi$ is the required extension of $\phi_0$.
	\end{proof}
	\begin{remark}
		The result above does not hold without the condition
		$R^1(\phi_0)_*\mathcal O_{X_0}=0$. Consider the projection
		$\pi_2:A_1\times A_2\to A_2$ for abelian varieties $A_1$ and $A_2$; note
		that
		$$
		R^1(\pi_2)_*\mathcal{O}_{A_1\times A_2}\neq0,
		$$
		since $\dim H^1(A_1,\mathcal O_{A_1})=\dim A_1>0$. In this case, the
		first higher direct image
		$R^1(\pi_2)_*\mathcal O_{A_1\times A_2}$
		is torsion free. A general deformation of $A_1\times A_2$ is a simple abelian
		variety $\mathcal A$ and therefore admits no nontrivial morphism to a
		lower-dimensional abelian variety.
	\end{remark}
	The extension theorem above can be applied to the steps of the K\"ahler minimal model program.
	\begin{theorem}[{\cite[Proposition 4.16]{HLR26} and \cite[Theorem 7.7]{HX26}}]\label{ExtMMP}
		Let $g:X\to S$ be a flat proper morphism with connected fibers from a
		generalized pair $(X,B+\bbeta)$ to a smooth, connected, relatively compact
		curve. Fix $0\in S$ and assume that $\operatorname{Supp}(B)$ does not
		contain $X_0$. Suppose that $(X_0,B_0+\bbeta_0)$ is K\"ahler and has
		canonical singularities, and that $\lfloor B_0\rfloor=0$, where
		$$
		\left.\left(K_X+X_0+B+\bbeta_X\right)\right|_{X_0}=K_{X_0}+B_0+\bbeta_0,
		$$
		and the negative part of the Boucksom--Zariski decomposition satisfies
		$$
		N\left(K_{X_0}+B_0+\bbeta_0\right) \wedge B_0=0 .
		$$
		Then every sequence of transcendental MMP-steps
		
		$$
		\left(X_0, B_0+\bbeta_0\right) \dashrightarrow\left(X_0^{(1)}, B_0^{(1)}+\bbeta_0^{(1)}\right) \dashrightarrow\left(X_0^{(2)}, B_0^{(2)}+\bbeta_0^{(2)}\right) \dashrightarrow \cdots
		$$
		extends to a sequence of $\left(K_X+B+\bbeta_X\right)$-negative proper bimeromorphic maps
		$$
		(X, B+\bbeta) / U \dashrightarrow \left(X^{(1)}, B^{(1)}+\bbeta^{(1)}\right) / U \dashrightarrow\left(X^{(2)}, B^{(2)}+\bbeta^{(2)}\right) / U \dashrightarrow \cdots,
		$$
		over some open neighborhood $U \subset S$ of $0$.
	\end{theorem}
	
	\section{Cohomological flatness and base change}\label{sec-cohomological-flatness}
	In this section, we establish a cohomological criterion for deformation invariance over a smooth curve, relating the constancy of fiberwise cohomology dimension to base change, extension, and the torsion freeness of higher direct image sheaves. To develop this framework, we use local complexes for higher direct images to present a series of equivalent conditions that characterize cohomological flatness over a one-dimensional Stein space. Furthermore, we introduce a new criterion for the torsion freeness of higher direct image sheaves over higher-dimensional bases in terms of higher Brill--Noether loci.
	\subsection{Local complexes for higher direct images}
	In this subsection, we use local complexes for higher direct images to give a
	sheaf-theoretic version of Proposition~\ref{Torzero}.
	\begin{definition}[{\cite[p. 123, Chapter~III]{BS76}}]
		Let $f:X \to Y$ be a proper morphism of complex spaces, $\mathscr{F}$ a coherent analytic sheaf on $X$ and $\mathcal{G}$ a coherent analytic sheaf on $Y$. Then the sheaf $\mathscr{F}$ is said to be \emph{transversal} on $\mathcal{G}$ if, for every $x \in X$,
		$$
		\Tor^{{\mathcal{O}}_{Y,y}}_i(\mathscr{F}_x,\mathcal{G}_y)=0, 
		$$
		for $i \geq 1, y=f(x)$.
	\end{definition}
	\begin{lemma}[{\cite[Theorem~4.1]{BS76}}]\label{main lemma}
		Let $f:X \to Y$ be a proper morphism of complex spaces, $\mathscr{F}$ a coherent analytic sheaf on $X$ and $y_0 \in Y$. Then there exist a neighborhood $V$ of $y_0$ and a complex $\mathcal{L}^\bullet$ of free sheaves of finite rank on $V$, which is bounded above, and which satisfies the following assertion:\par 
		For any coherent analytic sheaf $\mathcal{G}$ defined on an open subset of $Y$ such that $\mathscr{F}$ is transversal on it, there is an isomorphism
		$$
		R^\bullet f_*(\mathscr{F} \otimes_{\mathcal{O}_X} f^*(\mathcal{G})) \cong H^\bullet(\mathcal{L}^\bullet \otimes_{\mathcal{O}_Y} \mathcal{G})
		$$
		which is functorial in $\mathcal{G}$ and compatible with restrictions and
		short exact sequences in $\mathcal{G}$.\par
		Moreover, if $\mathscr{F}$ is flat over $Y$, then
		$\mathcal{L}^\bullet$ may be taken to be bounded.
	\end{lemma}
	\begin{remark}\label{main-rem}
		If $\mathscr{F}$ is flat over $Y$, then $\mathscr{F}_x$ is a flat
		$\mathcal{O}_{Y,y}$-module and is therefore transversal to every
		$\mathcal{G}$. Applying Lemma~\ref{main lemma} to sheaves of the form
		$\mathcal{O}_Y/\mathfrak{m}_y\mathcal{O}_Y$, we obtain a bounded complex
		$\mathcal{L}^\bullet$ of finite-rank free $\mathcal{O}_Y$-modules such that
		$$
		R^\bullet f_*(\mathscr{F}) \cong H^\bullet(\mathcal{L}^\bullet) \quad \textit{and} \quad H^\bullet(X_y,\mathscr{F}_{(y)}) \cong H^\bullet(\mathcal{L}^\bullet/\mathfrak{m}_y\mathcal{L}^\bullet) \cong H^\bullet(\mathcal{L}^\bullet \otimes_{\mathcal{O}_{Y,y}} \mathbb{C}(y))
		$$
		for every $y$ in the chosen neighborhood. Taking stalks at $y$ gives
		$$
		(R^\bullet f_*\mathscr{F})_y \cong  H^\bullet(L^\bullet) \quad \textit{and} \quad H^\bullet(X_y,\mathscr{F}_{(y)}) \cong H^\bullet(L^\bullet \otimes_{\mathcal{O}_{Y,y}} \mathbb{C}(y))
		$$
		where $L^\bullet:=\mathcal L^\bullet_y$ is the complex of free $\mathcal O_{Y,y}$-modules formed by the stalks of $\mathcal L^\bullet$.
	\end{remark}
	We now use this construction to promote Proposition \ref{Torzero} to the sheaf level.
	\begin{proposition} \label{jump-dimension}
		Let $f:X\to Y$ be a proper morphism of complex spaces where $Y$ is a smooth curve containing 0 and $\mathscr{F}$ a coherent analytic
		sheaf on $X$ that is flat over $Y$.
		Let $(\cO_{Y,0},\mathfrak{m}=(t))$ be the local ring of germs of
		analytic functions at $0$, which is a discrete valuation ring with
		uniformizer $t$ and residue field
		$\mathbb{C}(0)=\cO_{Y,0}/(t)=\CC$. Set
		$\mathscr{F}_{(t)}:=\mathscr{F}\otimes_{\mathcal{O}_X}
		\mathcal{O}_{X_t}$. For $q\geq0$, write $\cF^q:=R^qf_*\mathscr{F}$ and
		$h^q(X_t,\mathscr{F}_{(t)}):=\dim_{\mathbb{C}}
		H^q(X_t,\mathscr{F}_{(t)})$. Set
		$b_q=h^q(X_t,\mathscr{F}_{(t)})$ for nonzero $t\in Y$. Then:
		\begin{itemize}
			\item[(i)]\phantomsection\label{jump-cond-i} There is a noncanonical decomposition
			$$
			\cF_0^q \cong \mathcal{O}^{\oplus b_q}_{Y,0} \oplus T^q, \quad T^q:=\Tor_{\mathcal{O}_{Y
					,0}}(\cF^q_0) \cong \bigoplus_{i=1}^{\mu(T^q)}\mathcal{O}_{Y
				,0}/(t^{e_i}),
			$$
			where $\mu(T^q)=\dim_{\mathbb{C}}(T^q \otimes_{\mathcal{O}_{Y,0}} k(0))$ and $\ell(T^q):= \sum_{i=1}^{\mu(T^q)}e_i.$
			\item[(ii)]\phantomsection\label{jump-cond-ii} The dimensions satisfy
			$$
			h^q(X_0,\mathscr{F}_{(0)})-h^q(X_t,\mathscr{F}_{(t)})= \mu(T^q)+\mu(T^{q+1}).
			$$
			\item[(iii)]\phantomsection\label{jump-cond-iii} If moreover $\cF^{q+1}$ is locally free at $0$ and the degeneration is simple (all elementary divisors $e_i=1$, equivalently $\mathfrak{m}T^q=0$), then
			$$
			T^{q} \cong (\mathcal{O}_{Y,0}/(t))^{\oplus(h^q(X_0)-h^q(X_t))}, \quad \ell(T^q)=\mu(T^q)=h^q(X_0)-h^q(X_t).
			$$
		\end{itemize}
	\end{proposition}
	\begin{proof}
		\hyperref[jump-cond-i]{(i)} Since $f$ is proper and $\mathscr{F}$ is coherent,
		Theorem~\ref{direct image} shows that $\cF^q$ is a coherent
		$\mathcal{O}_{Y}$-module. At $0\in Y$, the stalk $\cF^q_0$ is a
		finitely generated $\mathcal{O}_{Y
			,0}$-module. Since $\mathcal{O}_{Y
			,0}$ is a discrete valuation ring, it is a principal ideal domain. Thus, by Proposition \ref{Torzero}
		, $\cF^q_0$ decomposes as
		$$
		\cF^q_0 \cong \mathcal{O}_{Y,0}^{\oplus r} \oplus T^q.
		$$
		Since the torsion module $T^q$ is annihilated by some power of the maximal
		ideal $\mathfrak{m}$, its support is concentrated at the origin; thus
		$T^q=0$ on $Y^* = Y \setminus\{0\}$. By Theorem~\ref{base change},
		$r=b_q=h^q(X_t,\mathscr{F}_{(t)})$ for general $t$.
		
		\hyperref[jump-cond-ii]{(ii)} Since $\mathcal{O}_{Y,0}$ is a discrete valuation ring,
		Theorem~\ref{UCT} and Remark~\ref{main-rem} give the exact sequence
		$$
		0 \to \cF^q_0 \otimes_{\mathcal{O}_{Y,0}}k(0) \to H^q(X_0,\mathscr{F}_{(0)}) \to \Tor^{\mathcal{O}_{Y,0}}_1(\cF^{q+1}_0,k(0)) \to 0.
		$$
		Taking dimensions in this exact sequence gives
		$$
		h^q(X_0,\mathscr{F}_{(0)})=\dim_{\mathbb{C}}H^q(X_0,\mathscr{F}_{(0)})=\dim(\cF^q_0 \otimes_{\mathcal{O}_{Y,0}}k(0))+\dim \Tor^{\mathcal{O}_{Y,0}}_1(\cF^{q+1}_0,k(0)).
		$$
		For $\dim(\cF^q_0 \otimes_{\mathcal{O}_{Y,0}}k(0))$, since $T^q$ has $\mu(T^q)$ cyclic generators $\mathcal{O}_{Y
			,0}/(t^{e_i})$ and $\mathcal{O}_{Y
			,0}/(t^{e_i}) \otimes \mathcal{O}_{Y
			,0}/(t) \cong \mathbb{C}$, each $\mathcal{O}_{Y
			,0}/(t^{e_i})$ contributes one dimension. The free summand contributes
		$b_q$ dimensions, so the decomposition
		$\cF_0^q \cong \mathcal{O}^{\oplus b_q}_{Y,0} \oplus T^q$ gives
		$$
		\dim(\cF^q_0 \otimes_{\mathcal{O}_{Y,0}}k(0))=b_q+\mu(T^q).
		$$
		Similarly, we can get
		$$
		\dim \Tor^{\mathcal{O}_{Y,0}}_1(\cF^{q+1}_0,k(0))=\mu(T^{q+1}).
		$$
		Consequently,
		$$
		h^q(X_0,\mathscr{F}_{(0)})=b_q+\mu(T^q)+\mu(T^{q+1}),
		$$
		and since $b_q=h^q(X_t,\mathscr{F}_{(t)})$ for general $t$, we obtain
		$$
		h^q(X_0,\mathscr{F}_{(0)})-h^q(X_t,\mathscr{F}_{(t)})=\mu(T^q)+\mu(T^{q+1}).
		$$
		\hyperref[jump-cond-iii]{(iii)} Since $\cF^{q+1}$ is locally free at $0$, we have $T^{q+1}=0$;
		thus \hyperref[jump-cond-ii]{(ii)} gives
		$$
		h^q(X_0,\mathscr{F}_{(0)})-h^q(X_t,\mathscr{F}_{(t)})=\mu(T^q).
		$$
		By \hyperref[jump-cond-i]{(i)}, 
		$$
		T^{q} \cong (\mathcal{O}_{Y,0}/(t))^{\mu(T^q)} \cong (\mathcal{O}_{Y,0}/(t))^{\oplus(h^q(X_0)-h^q(X_t))}.
		$$
		If the degeneration is simple, so that every elementary divisor satisfies
		$e_i=1$, then
		$$
		\ell(T^q):= \sum_{i=1}^{\mu(T^q)}e_i=\sum_{i=1}^{\mu(T^q)}1=\mu(T^q).
		$$
	\end{proof}
	\begin{remark}
		In particular, if we take $f$ to be a proper holomorphic submersion onto the unit disk, $\mathscr{F}:=\Omega(B_t)$ the sheaf of
		holomorphic sections of the smooth vector bundle $B_t$, then the dimension
		formula in \hyperref[jump-cond-ii]{(ii)} becomes the \emph{squeezing formula} in \cite[p. 68]{KS60}:
		$$
		h^{0,q}(t)+\nu^q(t)+\nu^{q-1}(t)=h^{0,q}(t_0),
		$$
		where $h^{0,q}(t)=\dim H^q(X_t,\Omega(B_t))$ and $\nu^q(t)$ is the
		number of eigenvalues of the Laplacian $\square_t$ that are smaller than a
		positive constant $c$.
	\end{remark}
	\subsection{A criterion for cohomological flatness}
	In this subsection, we let $f:X \to Y$ be a proper morphism of complex spaces, $\mathscr{F}$ a coherent analytic sheaf on $X$ which is flat with respect to $f$, $y$ a point in $Y$ and $q \geq 0$ an integer. \par
	First, we claim a corollary of \cite[Lemma~3.8]{HLR26} as a criterion for the conditions of the base change theorem.
	\begin{proposition}[{\cite[Lemma~3.8]{HLR26}}]\label{iso-sur-main}
		If $Y$ is Stein, then for any $y \in Y$, the base change map
		$$
		\varphi^q(y):(R^qf_*\mathscr{F})_y\otimes_{\mathcal{O}_{Y,y}}\mathbb{C}(y) \to H^q(X_y,\mathscr{F}_{(y)})
		$$
		is an isomorphism if and only if the restriction map 
		$$
		r^q: H^q(X,\mathscr{F}) \to H^q(X_y,\mathscr{F}_{(y)})
		$$
		is surjective.
	\end{proposition}
	\begin{proof}
		When $y \notin f(X)$, both sides are zero. Thus we may assume that
		$y\in f(X)$. Since $f$ is proper,
		Theorem~\ref{direct image} shows that $R^qf_*\mathscr{F}$ is coherent on
		$Y$. Since $Y$ is Stein, Cartan's Theorem~B gives
		$$
		H^p(Y,R^qf_*\mathscr{F})=0
		$$
		for all $p > 0$. Consider the Leray spectral sequence:
		$$
		E^{p,q}_2= H^p(Y,R^qf_*\mathscr{F}) \Rightarrow H^{p+q}(X,\mathscr{F}).
		$$
		Since all terms with $p>0$ vanish, the spectral sequence degenerates at the $E_2$ page. This induces a canonical isomorphism:
		$$
		\beta : H^0(Y,R^qf_*\mathscr{F}) \xrightarrow{\cong} H^q(X,\mathscr{F}).
		$$
		By Cartan's Theorem~A, the higher direct image sheaf $R^qf_*\mathscr{F}$ is globally generated on $Y$. Thus the evaluation map $\alpha_y$ in the following diagram is surjective:
		\[\begin{tikzcd}
			{H^0(Y,R^qf_*\mathscr{F})} && {(R^qf_*\mathscr{F})_y\otimes_{\mathcal{O}_{Y,y}}\mathbb{C}(y)} \\
			\\
			{H^q(X,\mathscr{F})} && {H^q(X_y,\mathscr{F}_{(y)})}
			\arrow["{\alpha_y}", from=1-1, to=1-3]
			\arrow["\beta"', from=1-1, to=3-1]
			\arrow["{\varphi^q(y)}", from=1-3, to=3-3]
			\arrow["{r^q}"', from=3-1, to=3-3].
		\end{tikzcd}\]
		Therefore $\varphi^q(y)$ is surjective if and only if $r^q$ is
		surjective. Corollary~\ref{iso-sur-lem} now shows that $\varphi^q(y)$ is an
		isomorphism if and only if $r^q$ is surjective.
	\end{proof}
	Extendible classes and jumping classes offer an alternative  framework for investigating cohomological flatness. In this subsection, we briefly recall their definitions based on the work of Griffiths. For a more comprehensive treatment, the reader is referred to \cite{Gr65}.\par
	Let $(K, \delta)$ be a graded complex over a graded Lie algebra complex $(A, d)$ and $\Phi(t) \in A^1\{t\}$ an integrable element. Define the \emph{deformed differential operator} $\delta_{\Phi}$ acting on formal power series by:
	\[
	\delta_{\Phi}\Gamma(t) = \delta\Gamma(t) - 2[\Phi(t), \Gamma(t)]. 
	\]
	For any cocycle $\gamma \in Z(K^q)$, if there exists a formal power series $\Gamma(t) \in K^q[[t]]$ such that $\Gamma(0) = \gamma$ and $\delta_{\Phi}\Gamma(t) = 0$, then the cocycle $\gamma$ is said to be \emph{extendible}. Furthermore, for a cohomology class $[\gamma] \in H^q(K)$, if there exists a cocycle $\gamma \in Z(K^q)$ representing this class that is itself extendible, then the cohomology class $[\gamma]$ is called an \emph{extendible class}. Denote by $E^q\subseteq H^q(K)$ the subspace of extendible classes.\par
	Furthermore, let $J^q$ be the subspace of $E^q$ consisting of extendible classes whose restrictions to $X_t$ vanish for $t \neq 0$; we call these \emph{jumping classes}. It is easy to see that if a class belongs to $J^q$, then the cocycle $Z_t$ representing it is the coboundary of a $(q-1)$-cochain with a pole at $t = 0$.\par
	We now proceed to the main theorem of this subsection, which gives a series of equivalent criteria for cohomological flatness.
	\begin{definition}[{\cite[Corollary~3.7, Chapter~III]{BS76}}]\label{coho-flat}
		We say $\mathscr F$ is \emph{cohomologically flat in dimension $q$ over $Y$} if the map $R^qf_*(\mathscr{F})_y \to R^qf_*(\mathscr{F}/\hat{\mathfrak{m}}_y\mathscr{F})_y$ is surjective and $R^qf_*(\mathscr{F})_y$ is a free module.
	\end{definition}
	
	The following criteria of cohomological flatness are probably known to the experts and we collect them here. 
	\begin{theorem}\label{Main}\label{l-Extension}
		If $Y$ is a smooth Stein curve, then the following conditions are equivalent:
		\begin{itemize}
			\item[(i)]\phantomsection\label{main-cond-i}
			$\mathscr{F}$ is cohomologically flat in dimension $q$ over $Y$; see
			Definition~\ref{coho-flat}.
			\item[(ii)]\phantomsection\label{main-cond-ii} The base change maps
			$$
			\varphi^q(y):(R^qf_*\mathscr{F})_y\otimes_{\mathcal{O}_{Y,y}}\mathbb{C}(y) \to H^q(X_y,\mathscr{F}_{(y)}), 
			$$
			$$
			\varphi^{q-1}(y):(R^{q-1}f_*\mathscr{F})_y\otimes_{\mathcal{O}_{Y,y}}\mathbb{C}(y) \to H^{q-1}(X_y,\mathscr{F}_{(y)})
			$$
			are isomorphisms.
			\item[(iii)]\phantomsection\label{main-cond-iii} The function
			$$
			y \mapsto \dim H^q(X_y,\mathscr{F}_{(y)})
			$$
			is locally constant.
			\item[(iv)]\phantomsection\label{main-cond-iv} The restriction maps  
			$$
			r^q: H^q(X,\mathscr{F}) \to H^q(X_y,\mathscr{F}_{(y)}),
			$$
			$$
			r^{q-1}: H^{q-1}(X,\mathscr{F}) \to H^{q-1}(X_y,\mathscr{F}_{(y)})
			$$
			are surjective.
			\item[(v)]\phantomsection\label{main-cond-v} The higher direct image sheaves $R^{q+1}f_*\mathscr{F}$ and $R^qf_*\mathscr{F}$ are torsion free.
			\item[(vi)]\phantomsection\label{main-cond-vi} The jumping classes satisfy
			$$
			J^{q+1}=0 \ \text{and}\ J^q=0.
			$$
		\end{itemize}
	\end{theorem}
	\begin{proof}
		\hyperref[main-cond-i]{(i)}$\iff$\hyperref[main-cond-ii]{(ii)} The
		equivalence follows from Definition~\ref{coho-flat} and
		Theorem~\ref{base change}.\par
		\hyperref[main-cond-i]{(i)}$\iff$\hyperref[main-cond-iii]{(iii)} This is just
		\cite[Theorem~4.12.(ii), Chapter~III]{BS76}.\par
		\hyperref[main-cond-iv]{(iv)}$\iff$\hyperref[main-cond-ii]{(ii)} This is
		just Proposition~\ref{iso-sur-main}.\par
		\hyperref[main-cond-v]{(v)}$\iff$\hyperref[main-cond-iii]{(iii)} If $h^q(X_y, \mathscr{F}_{(y)})$ is locally constant, by Proposition~\ref{jump-dimension}.\hyperref[jump-cond-ii]{(ii)}, we have
		$$
		\mu(T^q)+\mu(T^{q+1})=h^q(X_0,\mathscr{F}_{(0)})-h^q(X_t,\mathscr{F}_{(t)})=0.
		$$
		But $\mu(T^q)\geq 0$ and $\mu(T^{q+1}) \geq 0$, thus
		$$
		\mu(T^q)=0 \quad \text{and} \quad \mu(T^{q+1})=0.
		$$
		Then by Nakayama's lemma, we have
		$$ 
		T^q=0 \quad \text{and} \quad T^{q+1}=0,
		$$
		i.e., $R^{q}f_*\mathscr{F}$ and $R^{q+1}f_*\mathscr{F}$ are torsion free by Proposition~\ref{jump-dimension}.\hyperref[jump-cond-i]{(i)}. Conversely, if these two sheaves are torsion free, then
		$T^q=T^{q+1}=0$ in Proposition~\ref{jump-dimension}.\hyperref[jump-cond-ii]{(ii)}, and hence
		$$
		h^q(X_0,\mathscr{F}_{(0)})=h^q(X_t,\mathscr{F}_{(t)})
		$$ 
		by the dimension formula above; hence $h^q(X_y,\mathscr{F}_{(y)})$ is
		locally constant.\par
		\hyperref[main-cond-vi]{(vi)}$\iff$\hyperref[main-cond-iii]{(iii)} By \cite[Theorems 1.1 and 4.2]{Gr65}, we have
		$$
		H^q(X_0)/E^q \cong J^{q+1} \quad\text{and}\quad H^q(X_t) \cong E^q/J^q.
		$$
		Taking dimensions gives
		$$
		\dim H^q(X_0)-\dim E^q=\dim J^{q+1}
		\quad\text{and}\quad
		\dim H^q(X_t)=\dim E^q-\dim J^q.
		$$
		Therefore
		$$
		\dim H^q(X_0)-\dim H^q(X_t)=\dim J^{q+1}+\dim J^q.
		$$
		If $J^{q+1}=0$ and $J^q=0$, then
		$$
		\dim H^q(X_0)=\dim H^q(X_t),
		$$
		so the function $y\mapsto\dim H^q(X_y)$ is locally constant. Conversely,
		if this function is locally constant, then
		$$
		\dim J^{q+1}+\dim J^q=0.
		$$
		Thus $J^{q+1}=0$ and $J^q=0$.
	\end{proof}
	\begin{remark}\label{rem-main-criterion}
		\begin{enumerate}[(a)]
			\item\label{rem-main-criterion-a} The proof shows that no hypothesis beyond $Y$ being a complex space is required to deduce the remaining statements from \hyperref[main-cond-i]{(i)}. The implication \hyperref[main-cond-ii]{(ii)}$\Rightarrow$\hyperref[main-cond-i]{(i)} requires only that $Y$ be reduced; \hyperref[main-cond-iv]{(iv)}$\Rightarrow$\hyperref[main-cond-ii]{(ii)} requires only that $Y$ be Stein; and \hyperref[main-cond-v]{(v)}$\Rightarrow$\hyperref[main-cond-iii]{(iii)} requires only that $Y$ be a smooth one-dimensional complex space.
			\item An alternative direct proof of \hyperref[main-cond-iii]{(iii)}$\Leftrightarrow$\hyperref[main-cond-v]{(v)}, based on \cite[Proposition~4.12]{RT1}, is given in the appendix.
			\item 
			Theorem~\ref{Main} says that the following six conditions are equivalent.
			More precisely, its proof follows the directed paths displayed below; in
			particular, the converse implication from torsion freeness is obtained by
			first passing through $q$-invariance.
			\[
			\begin{tikzcd}[
				column sep=4.1em,
				row sep=3.6em,
				cells={nodes={
						font=\small,
						inner xsep=6pt,
						inner ysep=5pt,
						rounded corners=2.5pt,
						fill=white
				}},
				every arrow/.append style={
					line width=0.5pt,
					shorten <=3pt,
					shorten >=3pt
				}
				]
				|[draw=Gray!55, fill=Gray!7]| {\substack{\text{\hyperref[main-cond-iv]{(iv)} }q\text{-extension}\\(q-1)\text{-extension}}}
				& |[draw=Blue!70, fill=Blue!10, line width=0.7pt]| {\substack{\text{\hyperref[main-cond-ii]{(ii)} }q\text{-base change}\\(q-1)\text{-base change}}}
				& |[draw=Gray!55, fill=Gray!7]| {\text{\hyperref[main-cond-i]{(i)} }q\text{-flatness}}
				\\
				{}
				& |[draw=Blue!55, fill=Blue!5]| {\substack{\text{\hyperref[main-cond-v]{(v)} }(q+1)\text{-torsion freeness}\\q\text{-torsion freeness}}}
				& |[draw=Blue!70, fill=Blue!10, line width=0.7pt]| {\text{\hyperref[main-cond-iii]{(iii)} }q\text{-invariance}}
				\\
				{}
				& {}
				& |[draw=Gray!55, fill=Gray!7]| {\substack{\text{\hyperref[main-cond-vi]{(vi)} }(q+1)\text{-jump vanishing}\\q\text{-jump vanishing}}}
				\arrow[<->, shift left=5pt, "{\scriptscriptstyle\mathrm{Prop.}\,\ref{iso-sur-main}}", from=1-1, to=1-2]
				\arrow[<->, shift right=5pt, from=1-1, to=1-2]
				\arrow[<->, "{\scriptscriptstyle\mathrm{Def.}+\mathrm{BC}}", from=1-2, to=1-3]
				\arrow[<->, "{\text{\cite{BS76}}}", from=1-3, to=2-3]
				\arrow[<->, shift left=5pt, from=1-2, to=2-2]
				\arrow[<->, shift right=5pt, "{\text{universal coefficient thm}}"', from=1-2, to=2-2]
				\arrow[<->,"{\ref{jump-dimension}\ \text{or}\ \text{\cite{RT1}}}", from=2-2, to=2-3]
				\arrow[<->, "{\text{\cite{Gr65}}}", from=2-3, to=3-3]
			\end{tikzcd}
			\]
		\end{enumerate}
		\noindent Here ``Def.+BC'' denotes Definition~\ref{coho-flat} together with Theorem~\ref{base change}. The upper route is justified by Proposition~\ref{iso-sur-main} and \cite[Theorem~4.12.(ii), Chapter~III]{BS76}. The lower route uses Proposition~\ref{jump-dimension}; the last edge uses \cite[Theorems~1.1 and~4.2]{Gr65}. An alternative direct proof of \hyperref[main-cond-iii]{(iii)}$\Leftrightarrow$\hyperref[main-cond-v]{(v)} is given in Appendix~\ref{app-base change}.
	\end{remark}
	\begin{corollary}\label{l-squeeze}
		If $\dim_{\mathbb{C}}H^{q+1}(X_y,\mathscr{F}_{(y)})$ and $\dim_{\mathbb{C}}H^{q-1}(X_y,\mathscr{F}_{(y)})$ are constant, then $\dim_{\mathbb{C}}H^{q}(X_y,\mathscr{F}_{(y)})$ is constant.
	\end{corollary}
	\begin{proof}
		By Theorem \ref{Main}, if $\dim_{\mathbb{C}}H^{q+1}(X_y,\mathscr{F}_{(y)})$ is constant, $\varphi^{q+1}$ and $\varphi^q$ are isomorphisms. Similarly, if $\dim_{\mathbb{C}}H^{q-1}(X_y,\mathscr{F}_{(y)})$ is constant, $\varphi^{q-1}$ and $\varphi^{q-2}$ are isomorphisms. Therefore, since $\varphi^{q}$ and $\varphi^{q-1}$ are isomorphisms, $\dim_{\mathbb{C}}H^{q}(X_y,\mathscr{F}_{(y)})$ is constant.
	\end{proof}
	Finally, the following result from \cite{HLR26} illustrates the importance of
	the local-constancy hypothesis for $h^q(X_y,\mathscr{F}_{(y)})$.
	\begin{proposition}[{\cite[Corollary~3.13]{HLR26}}]\label{H}
		Let $f : X \to \Delta$ be a flat proper morphism from a normal complex analytic variety onto the unit disk in $\mathbb{C}$, and $L$ a line bundle on $X$. Then the following are equivalent:
		\begin{itemize}
			\item[(a)] $L|_{X_t}$ is globally generated for any $t \in V$, where $V \subset \Delta$ is a nonempty Zariski open subset;
			\item[(b)] There exists a point $s_0 \in \Delta$ such that the base change map
			$$
			\varphi^0(s_0) : (f_*L)(s_0) \to H^0(X_{s_0}, L|_{X_{s_0}})
			$$
			is surjective, and $L|_{X_{s_0}}$ is globally generated. 
		\end{itemize}
	\end{proposition}
	We now give an example showing that in Proposition \ref{H} the global generation of one restricted line bundle would not  propagate  without the base change (or  
	local constancy)  hypothesis. Denote by \(\operatorname{Pic}^{k}(C)\) the set of isomorphism classes of holomorphic line bundles of degree \(k\) on the compact Riemann surface \(C\).
	\begin{example}\label{egH}
		Let $C$ be a smooth projective curve of genus two and $\mathcal{P}$ the
		normalized \emph{Poincar\'e line bundle of degree two} on
		$C\times\Pic^2(C)$ (cf. \cite[p. 166]{ACGH85}). Choose a smooth analytic curve $S\subset\Pic^2(C)$ through
		$[K_C]$, together with an embedding
		$$
		\phi: S \hookrightarrow \Pic^2(C)
		$$
		such that $\phi(0)=[K_C]\in \Pic^2(C)$. It induces the map
		$$
		\Phi = \operatorname{id}_C \times \phi: C \times S \longrightarrow C \times \Pic^2(C).
		$$ 
		Set
		$X_s:=C\times\{s\}$. The map $\Phi$ gives the line bundle
		$$
		L=\Phi^*(\mathcal{P})
		$$
		on $C\times S$, with $L_s=L|_{X_s}\in\Pic^2(C)$, $L_0=K_C$, and
		$L_s\neq K_C$ for $s\neq0$. Since $C$ has genus two, it is hyperelliptic,
		and $K_C$ induces the map
		$$
		\phi_{K_C}: C \to \mathbb{P}^1.
		$$
		Thus $|K_C|$ is base point free; equivalently, $K_C$ is globally generated, and
		$$
		h^0(C,L_0)=h^0(C,K_C)=g=2.
		$$
		When $s\neq0$, the Riemann--Roch theorem gives
		$$
		h^0(C,L_s)-h^1(C,L_s)=\deg(L_s)+1-g=1.
		$$
		By Serre duality, we have
		$$
		h^1(C,L_s)=h^0\!\left(C,\mathcal O_C(K_C-L_s)\right).
		$$
		Since $\mathcal O_C(K_C-L_s)$ is nontrivial of degree $0$, it has no nontrivial global sections. Thus
		$$
		h^1(C,L_s)=h^0\!\left(C,\mathcal O_C(K_C-L_s)\right)=0.
		$$
		Therefore,
		$$
		h^0(C,L_s)=1.
		$$
		Thus every nonzero global section is a scalar multiple of a fixed section.
		Since $\deg(L_s)=2$, the zero divisor of this section contains two points
		(counted with multiplicity), which belong to the base locus $\BS(L_s)$.
		Hence $L_s$ is not globally generated for $s\neq0$.
	\end{example}
	\begin{remark} Example \ref{egH} also shows that global generation of one restricted line bundle does not imply the local constancy of $h^0(X_s,L_s)$.  
	\end{remark}
	
	We conclude this subsection by stating Koll\'ar's~\cite{Kollar86} torsion-freeness theorem, in the analytic form of \cite[Corollary~3.15]{Naka87}. 
	
	\begin{proposition}[{\cite[Corollary~3.15]{Naka87}}]
		\label{p-kollar-torsion freeness}
		Let $f:X\to S$ be a proper K\"ahler morphism from a normal complex space
		with canonical singularities to a smooth complex curve. If $K_X$ is \emph{$f$-semiample}, i.e., 
		there exists an integer $r>0$ such that $rK_X$ is
		Cartier and the natural evaluation morphism
		\[
		f^*f_*\mathcal O_X(rK_X)\longrightarrow\mathcal O_X(rK_X)
		\]
		is surjective, then
		$R^i f_*\mathcal O_X(mK_X)$ is torsion free, and
		hence locally free on $S$, for all integers $i\ge 0,m\ge 1$.
	\end{proposition}
	
	\subsection{Brill--Noether loci and torsion freeness}\label{app-brill-noether}
	In the final part of this section, we give a criterion for cohomological flatness. We emphasize the assumption that $Y$ is one-dimensional in Remark~\ref{rem-main-criterion}.\hyperref[rem-main-criterion-a]{(a)} and Appendix~\ref{app-base change}. We first give an example showing that this assumption is necessary and then give a new criterion for the torsion freeness of higher direct image sheaves over higher-dimensional bases in terms of higher Brill--Noether loci. This criterion is mainly inspired by \cite{Teleman15}.\par
	First, we introduce the relation between the torsion freeness of a coherent analytic sheaf and its vanishing.
	\begin{lemma} \label{l-codim-support-vanishing}
		Let $X$ be a locally irreducible complex analytic space and $\mathscr{F}$ a coherent analytic sheaf on $X$. If $\codim\Supp(\mathscr{F}) \geq 1$, then $\Supp(\mathscr{F}) \subset \Supp(T(\mathscr{F}))$. Furthermore, if $\mathscr{F}$ is torsion free, then $\mathscr{F}=0$.
	\end{lemma}
	\begin{proof}
		The support $Z=\Supp(\mathscr{F})$ is a closed analytic subset of $X$.
		If $\codim Z\geq1$, then for each $x\in Z$ we can choose a nonzero
		$f\in\mathcal{O}_{X,x}$ that vanishes on the germ $Z_x$. By R\"uckert's
		Nullstellensatz \cite[p. 67]{GR84}, some power $f^k$ annihilates every
		element of $\mathscr{F}_x$. Since $\mathcal{O}_{X,x}$ is an integral domain,
		$f^k\neq0$, so every element of $\mathscr{F}_x$ is a torsion element. Hence
		$\Supp(\mathscr{F})\subset\Supp(T(\mathscr{F}))$. If $\mathscr{F}$ is
		torsion free, the latter support is empty, and therefore $\mathscr{F}=0$.
	\end{proof}
	Teleman states a criterion for torsion freeness of $R^1\pi_*(\mathcal{E})$:
	\begin{proposition}[{\cite[Corollary~1.4]{Teleman15}}]\label{Teleman}
		Let $B,M$ be complex manifolds, $\pi:M\to B$ a proper surjective
		holomorphic submersion, and $\mathcal{E}$ a holomorphic vector bundle on
		$M$. Denote by $\mathcal{E}_x$ the restriction of $\mathcal{E}$ to the
		fiber $M_x$, regarded as a locally free sheaf on $M_x$. Define the
		\emph{Brill--Noether locus} by
		$$
		\BN=\{x\in B\ |\ h^0(M_x, \mathcal{E}_x)\neq 0\}.
		$$
		If $\pi_*(\mathcal{E})=0$ and $\codim_B(\BN) \geq 2$, then $R^1\pi_*(\mathcal{E})$ is torsion free.
	\end{proposition} 
	We now use Teleman's criterion to construct a counterexample showing that Theorem~\ref{Main} would fail when the base is not one-dimensional. 
	\begin{example}\label{Ex}
		Let $C$ be a smooth projective curve of genus two and
		$J=\mathrm{Pic}^0(C)$ its Jacobian. Write
		$$
		f : C \times J \longrightarrow J
		$$
		for the projection and $\mathcal{P}$ for a normalized Poincar\'e line bundle. Then $f$ is smooth projective and $J$ is a nonsingular projective surface. For $x \neq 0$, $L_x \neq \mathcal{O}_C$ is a divisor of degree $0$. It has no global section, thus $h^0(C,L_x)=0$. By Proposition \ref{p-pushforward-torsion free}, $f_*\mathcal{P}$ is torsion free; since it vanishes away from the origin, Lemma \ref{l-codim-support-vanishing} gives
		$$
		f_*\mathcal{P}=0.
		$$
		When $x=0$, $L_0=\mathcal{O}_C$ has only constant global sections, so $h^0(C,L_0)=1$. Thus the Brill--Noether locus contains only the single point $x=0$. Since $J$ is two-dimensional, $\codim(\BN)=2$. Then, by Proposition \ref{Teleman}, $R^1f_*\mathcal{P}$ is torsion free.
		By the Riemann--Roch theorem,
		$$
		h^0(C,L_0)-h^1(C,L_0)=\deg(L_0)+1-g=-1.
		$$
		Thus $h^1(C,L_0)=2$. Consider the base change map
		$$
		\varphi^0(0): (f_*\mathcal{P})\otimes k(0) \to H^0(C,L_0).
		$$
		Its source is zero, whereas the target is one-dimensional, so the map is not surjective.
	\end{example}
	\begin{remark}\label{rem-higher-direct-images}
		Since every fiber above is the one-dimensional curve $C$, Grothendieck's vanishing theorem (\cite[Theorem~2.7]{Hart77}) gives
		$$
		R^if_*\mathcal{P}=0
		$$
		for all $i \geq 2$. However, for higher-dimensional fibers, for example, let the base be $B=\Pic^1(C) \times \Pic^1(D)$, where $C,D$ are smooth projective curves of genera $g_1$ and $g_2$; then each fiber is the two-dimensional surface $X=C\times D$. Thus $R^2f_*\mathcal{P} \neq 0$, and the criterion above fails to establish its torsion freeness. We next generalize Proposition \ref{Teleman} to higher direct images using regular sequences, Cohen--Macaulay rings, and the idea of \emph{complete intersections} (cf. \cite[Chapter~II, Exercise~8.4]{Hart77}).
	\end{remark}
	For the remainder of this subsection, let $B$ and $M$ be complex analytic spaces, $\pi:M\to B$ a proper flat morphism and $\mathcal E$ a holomorphic vector bundle on $M$.
	\begin{definition}
		We define the \emph{$k$-th Brill--Noether locus} of the pair $(\pi, \mathcal{E})$ as
		$$
		\BN^k_\pi(\mathcal E):=\{x\in B\mid h^k(\mathcal E_x)\neq0\}\subset B,
		$$
		where, for $x \in B$, $\mathcal{E}_x$ denotes the restriction of $\mathcal{E}$ to the fiber $M_x := \pi^{-1}(x)$ and $h^k(\mathcal{E}_x):=\dim_{\mathbb{C}}H^k(M_x,\mathcal{E}_x)$.
	\end{definition}
	\begin{theorem}\label{cri}
		Assume that $B$ is Cohen--Macaulay (e.g., $B$ has rational
		singularities; cf.\ \cite[p. 51]{KKMSD73}) and that, for every $k\geq 1$ and $i$ with 
		\[
		\max\{k-\dim B,0\}\leq i\leq k-1,
		\]
		we have $R^{i}\pi_*(\mathcal{E})=0$ and
		$\operatorname{codim}_{B}(\BN_\pi^i)\geq k-i+1$. Then
		$R^k\pi_*(\mathcal{E})$ is torsion free.
	\end{theorem}
	\begin{proof}
		Fix $x\in B$ and suppose, for a contradiction, that $R^k\pi_*(\mathcal E)$ has torsion at $x$. By
		\cite[Proposition~7.1]{Mat89}, there is a nonzero element
		\[
		u\in (R^k\pi_*(\mathcal{E}))_x
		\]
		whose annihilator $\mathfrak p:=\operatorname{Ann}_R(u)$ is a prime ideal. Since $u$
		is a torsion element, $\mathfrak p\neq(0)$; set $h:=\operatorname{ht}(\mathfrak p)\geq1$.
		Let $I_{\BN}\subset R$ be the ideal defining the germ at $x$ of
		$\BN_\pi^{k-h}(\mathcal{E})$. The codimension hypothesis gives
		\[
		\operatorname{ht}(I_{\BN})\geq h+1>\operatorname{ht}(\mathfrak p),
		\]
		so $I_{\BN}\nsubseteq \mathfrak p$. Choose $f\in I_{\BN}\setminus \mathfrak p$. After localizing
		at $\mathfrak p$, the element $f$ becomes a unit, and hence
		$(I_{\BN})_{\mathfrak p}=R_{\mathfrak p}$. Thus the Brill--Noether locus is empty after localization
		at $\mathfrak p$. By \cite[Corollary~3.5, Chapter~III]{BS76}, the vanishing of $H^{k-h}(M_b,\mathcal E_b)$ at a point $b\in B$ implies that $R^{k-h}\pi_*\mathcal E$ vanishes on a neighborhood of $b$. Consequently,
		\[
		\operatorname{Supp}\bigl(R^{k-h}\pi_*\mathcal E\bigr)\subseteq
		\BN^{k-h}_\pi(\mathcal E).
		\]
		Since $I_{\BN}\nsubseteq \mathfrak p$, this support becomes empty after localization at $\mathfrak p$, and therefore
		\[
		\bigl((R^{k-h}\pi_*(\mathcal{E}))_x\bigr)_{\mathfrak p}=0.
		\]
		Since $\mathcal{O}_{B,x}$ is Cohen--Macaulay, there exists a regular
		sequence $(\varphi_1,\ldots,\varphi_h)$ in $\mathfrak p$, represented by holomorphic
		functions near $x$. Define $I_j=(\varphi_1,\ldots,\varphi_j)$ and
		$\mathcal{E}_j=\mathcal{E}/(\pi^*I_j)\mathcal{E}$. Then necessarily
		\[
		\bigl((R^{k-h}\pi_*(\mathcal{E}_h))_x\bigr)_{\mathfrak p}=0.
		\]
		It therefore suffices to produce a nonzero element in this localized module;
		this will give the desired contradiction.
		\par
		Since $(\varphi_1,\ldots,\varphi_h)$ is a regular sequence,
		multiplication by $\varphi_j$ is injective on
		$\mathcal{O}_{B,x}/I_{j-1}$. Its cokernel gives the exact sequence
		$$
		0 \to \mathcal{O}_{B,x}/I_{j-1} \xrightarrow{\cdot \varphi_j} \mathcal{O}_{B,x}/I_{j-1} \to \mathcal{O}_{B,x}/I_j \to 0. 
		$$
		Since $\pi$ is flat and $\mathcal{E}$ is locally free, applying
		$-\otimes_{\mathcal{O}_B}\mathcal{E}$ to the sequence above gives the
		short exact sequence
		$$
		0 \to \mathcal{E}_{j-1} \xrightarrow{\varphi_j} \mathcal{E}_{j-1} \to \mathcal{E}_j \to 0.
		$$
		For $j=1$, the sequence
		$$
		0 \to \mathcal{E}_{0} \xrightarrow{\varphi_1} \mathcal{E}_{0} \to \mathcal{E}_1 \to 0,
		$$
		yields
		$$
		\cdots \to R^i\pi_*(\mathcal{E}_0) \to R^i\pi_*(\mathcal{E}_0) \to R^i\pi_*(\mathcal{E}_1) \to R^{i+1}\pi_*(\mathcal{E}_0) \to \cdots. \quad (1)
		$$
		When $\max\{{k-\dim B},0\}\leq i \leq k-2$, by our hypothesis, $$R^i\pi_*(\mathcal{E}_0)=R^{i+1}\pi_*(\mathcal{E}_0)=0.$$ 
		Thus $R^i\pi_*(\mathcal{E}_1)=0$ for all $i \leq k-2$. Similarly, we have $R^i\pi_*(\mathcal{E}_j)=0$ for all $i \leq k-j-1$.
		For the sequence (1) with $i=k-1$, we have
		$$0 \to (R^{k-1}\pi_*(\mathcal{E}_1))_x \xrightarrow{\delta}(R^k\pi_*(\mathcal{E}_0))_x \xrightarrow{\varphi_1}(R^k\pi_*(\mathcal{E}_0))_x
		$$
		Since $\varphi_1\in \mathfrak p=\Ann(u)$, we have $\varphi_1u=0$ and hence
		$u\in\ker(\varphi_1)=\im(\delta)$. Thus there exists a unique
		$u_1\in(R^{k-1}\pi_*(\mathcal{E}_1))_x$ such that $\delta(u_1)=u$.
		Since $\delta$ is injective, $\Ann(u_1)=\Ann(u)=\mathfrak p$. Repeating this process
		$h$ times produces an element
		$u_h\in(R^{k-h}\pi_*(\mathcal{E}_h))_x$ with $\Ann(u_h)=\mathfrak p$, completing
		the proof.
	\end{proof}
	Now we can use our generalized criterion to settle the question raised in Remark \ref{rem-higher-direct-images}.
	\begin{example}
		Let $C$ be a smooth projective curve of genus $g_1\geq3$ and $D$ one of
		genus $g_2\geq3$. Then $B=\Pic^1(C)\times\Pic^1(D)$ is a torsor under the abelian variety $J(C)\times J(D)$ of
		dimension $g_1+g_2$. Let $M=C\times D\times B$ and $\pi:M\to B$ the
		proper holomorphic submersion; each fiber is the surface $X=C\times D$.
		Let $\mathcal{P}_C$ be the normalized Poincar\'e line bundle on
		$C\times\Pic^1(C)$ and $\mathcal{P}_D$ the corresponding bundle on
		$D\times\Pic^1(D)$, and
		set $\mathcal{E}=\mathcal{P}_C\boxtimes\mathcal{P}_D$. Then, for every
		$b=(L_1,L_2)\in B$, $\mathcal{E}_b=L_1\boxtimes L_2$. \par
		We apply Theorem~\ref{cri} with $k=2$. We first check the vanishing of
		$R^0\pi_*(\mathcal{E})$ and $R^1\pi_*(\mathcal{E})$. Decompose $\pi$ as $\pi_C: C \times \Pic^1(C) \to \Pic^1(C)$ and $\pi_D: D \times \Pic^1(D) \to \Pic^1(D)$. Denote the \emph{Brill--Noether variety} by
		$$
		W_d^r(C)=\{{L\in \Pic^d(C)|h^0(C,L)\geq r+1\}},
		$$
		By \cite[Chapter~5]{HM98}, we have
		$$
		\dim W_1^0(C)=1 \quad \text{and} \quad \dim W_1^0(D)=1.
		$$
		From 
		$$
		\Supp(R^0\pi_{C*}(\mathcal{P}_C)) \subset W_1^0(C),
		$$
		we know that $\codim \Supp(R^0\pi_{C*}(\mathcal{P}_C)) \geq \codim W_1^0(C)=g_1-1 \geq 2.$ And by Proposition \ref{p-pushforward-torsion free}, we know that $R^0\pi_{C*}(\mathcal{P}_C)$ is torsion free. Then by Lemma \ref{l-codim-support-vanishing}, we get
		$$
		R^0\pi_{C*}(\mathcal{P}_C)=0.
		$$
		That means
		$$
		\mathbb{R}\pi_{C*}(\mathcal{P}_C) \cong R^1\pi_{C*}(\mathcal{P}_C)[-1].
		$$
		Similarly, for curve $D$, we have
		$$
		\mathbb{R}\pi_{D*}(\mathcal{P}_D) \cong R^1\pi_{D*}(\mathcal{P}_D)[-1].
		$$
		Thus
		$$
		\mathbb{R}\pi_*(\mathcal{E}) \cong (R^1\pi_{C*}(\mathcal{P}_C)[-1]) \boxtimes^L (R^1\pi_{D*}(\mathcal{P}_D)[-1]) \cong (R^1\pi_{C*}(\mathcal{P}_C) \boxtimes^L R^1\pi_{D*}(\mathcal{P}_D))[-2].
		$$
		That means 
		$$
		R^0\pi_*(\mathcal{E})=0 \quad \text{and} \quad R^1\pi_*(\mathcal{E})=0.
		$$
		We now check the codimension conditions for $\BN^0_\pi$ and
		$\BN^1_\pi$. 
		When $i=0$, $h^0(X,\mathcal{E}_b) \neq 0$ if and only if
		$h^0(C,L_1)\geq 1$ and $h^0(D,L_2) \geq 1$. Thus
		$$
		\BN_\pi^0= W_1^0(C) \times W_1^0(D),
		$$ 
		which is of dimension two.
		Since $\dim B = g_1 +g_2$, 
		$$
		\codim(\BN_\pi^0)=g_1+g_2-2 \geq 4>3=k-i+1. 
		$$
		When $i=1$, from
		$$
		H^1(X,\mathcal{E}_b)=(H^0(C,L_1)\otimes H^1(D,L_2)) \oplus (H^1(C,L_1)\otimes H^0(D,L_2)),
		$$
		we have $h^1(X,\mathcal{E}_b) \neq 0$ if and only if
		$h^0(C,L_1) \geq 1$ or $h^0(D,L_2) \geq 1$. Thus
		$$
		\BN^1_\pi=(W_1^0(C) \times \Pic^1(D)) \cup (\Pic^1(C) \times W_1^0(D)).
		$$
		Then
		$$
		\codim(\BN_\pi^1)=g_1+g_2-\max\{g_1+1, g_2+1\}=\min\{g_1-1, g_2-1\} \geq 2=k-i+1.
		$$
		Theorem~\ref{cri} now shows that $R^2\pi_*(\mathcal{E})$ is torsion free.
	\end{example}
	\begin{remark}
		In fact, by Lemma \ref{l-codim-support-vanishing}, we know that under the codimension condition
		$$
		\operatorname{codim}_{B}(\BN_\pi^i) \geq k-i+1,
		$$ 
		$R^i\pi_*(\mathcal{E})=0$ if and only if $R^i\pi_*(\mathcal{E})$ is torsion free. Thus if we have gotten $\pi_*\mathcal{E}$ is torsion free, to get $R^k\pi_*(\mathcal{E})$ is torsion free, we only need to check the codimention conditions for each $i \leq k-1$. This also provides another method to check the example above.  
	\end{remark}
	Indeed, the codimension condition
	$$
	\operatorname{codim}_{B}(\BN_\pi^i) \geq k-i+1
	$$ 
	is optimal, as shown in the following simple counterexample.
	\begin{example}
		Let $A$ be an abelian variety of dimension $g$ and $\hat{A}=\Pic^0(A)$ its dual.
		Let $M=A\times\hat{A}$, $\mathcal{P}$ its Poincar\'e line bundle, and
		$\pi:A\times\hat{A}\to\hat{A}$ the projection. When $\hat{x}\neq0$, we have
		$$
		h^i(A,\mathcal{P}_{\hat{x}})=0
		$$ 
		by \cite[(vii), p. 76]{Mum70}. When $\hat{x}=0$, we have
		$$
		h^i(A,\mathcal{P}_0)=h^i(A,\mathcal{O}_A)
		=\dim(\Lambda^i\mathbb{C}^g)=\binom{g}{i}.
		$$
		Thus, for all $i=0,1,\ldots,g$, we have
		$$
		\BN^i_{\pi}=\{0\}.
		$$
		Then 
		$$
		\codim(\BN_\pi^i)=g
		$$
		for all $0 \leq i \leq g$.
		The Fourier--Mukai transform \cite{Mukai81} gives
		$$
		\mathbb{R}\pi_*{\mathcal{P}}=\mathcal{O}_{\{0\}}[-g],
		$$
		and hence
		$$
		R^i\pi_*\mathcal{P}=0 \quad \text{for all $0\leq i<g$},
		\qquad R^g\pi_*\mathcal{P}=\mathcal{O}_{\{0\}}.
		$$
		Thus $R^g\pi_*\mathcal{P}$ is a torsion sheaf, showing that the
		codimension condition in the theorem is optimal.
	\end{example}

	\section{Deformation invariance of plurigenera}\label{sec-invariance}
	We now apply the preceding criteria (Theorem \ref{Main}) to establish the deformation invariance of plurigenera. Throughout this section, $S$ denotes a relatively compact, smooth, connected complex curve. When a central fiber is specified, we fix a point $0\in S$ and write $X_0=f^{-1}(0)$. In arguments local near $0$, we may replace $S$ by a relatively compact neighborhood of $0$ and restrict the family accordingly. 
	\subsection{Invariance of plurigenera for K\"ahler threefolds}
	We begin our discussion with families of K\"ahler threefolds.
	\begin{theorem}\label{t-threefold-plurigenera}
		Let $f:X\to S$ be a $\mathbb{Q}$-factorial K\"ahler family of compact threefolds over a relatively compact, smooth connected curve $S$. Assume that
		$X_0$ has canonical singularities. Then, for every integer $m\geq1$, the
		$m$-th plurigenus
		\[
		P_m(X_t)=\dim H^0(X_t,mK_{X_t})
		\]
		is independent of $t$ near $0$.
	\end{theorem}
	
	\begin{proof}
		The proof follows the strategy of \cite[Theorem~11]{Naka85}. By
		\cite[Theorem~1.2]{DHP23}, we can run a $K_X$-MMP over $S$. This MMP terminates with
		either a relative minimal model or a relative Mori fiber space. First,
		suppose that some step of the MMP contracts the central fiber. We claim that
		$P_m(X_0)=0$ for all $m\geq1$. Suppose that $X_0$ is contracted at the
		$i$-th step $\phi^i:X^{(i)}\dashrightarrow X^{(i+1)}/S$ of the MMP. This step must
		be a divisorial contraction because $X_0$ is a divisor. Then $-K_{X^{(i)}}$ is
		$\phi^i$-ample, so
		$H^0(X_0^{(i)},\mathcal O_{X_0^{(i)}}(mK_{X^{(i)}}))=0$ for all $m\geq1$. Since the
		composition of the first $i$ steps of the MMP does not contract $X_0$, the
		fibers $X_0$ and $X_0^{(i)}$ are bimeromorphic. Hence
		$P_m(X_0)=P_m(X_0^{(i)})=0$ for all $m\geq1$. Since $f$ is flat, upper
		semicontinuity of $h^0$ implies that $P_m(X_t)=0$ for every nearby $t$ and
		every $m\geq1$.
		
		Therefore, we may assume that the central fiber is not contracted by any step of the MMP throughout the rest of the proof, and we divide the proof into two cases.
		
		\noindent\textbf{Case 1:} The relative $K_X$-MMP terminates with a relative minimal model over $S$.
		\begin{center}
			\begin{tikzcd}[ampersand replacement=\&]
				X \&\& {X'} \\
				\& S
				\arrow["\phi", dashed, from=1-1, to=1-3]
				\arrow["f"', from=1-1, to=2-2]
				\arrow["{f'}", from=1-3, to=2-2]
			\end{tikzcd}
		\end{center}
		
		Since the MMP terminates with a relative minimal model, $K_{X'}$ is nef over
		$S$. By \cite{DO24,DO26}, $K_{X'}$ is semiample over $S$.
		Proposition~\ref{p-kollar-torsion freeness} shows that
		$R^if'_*\mathcal O_{X'}(mK_{X'})$ is locally free on $S$ for all $i\geq 0$
		and $m\geq 1$. By Theorem~\ref{Main}, the degree-zero base change map
		\[
		\varphi^0(t):
		(f'_*\mathcal{O}_{X'}(mK_{X'}))_t
		\otimes_{\mathcal{O}_{S,t}}\mathbb{C}(t)
		\longrightarrow
		H^0\!\left(X'_t,\mathcal{O}_{X'}(mK_{X'})|_{X'_t}\right)
		\]
		is surjective and hence an isomorphism. Moreover, by
		\cite[Corollary~6.2]{Naka87},
		\[
		\mathcal{O}_{X_t'}(mK_{X'}) \cong \mathcal{O}_{X_t'}(mK_{X_t'}),
		\qquad \text{for all } t \in S,\ m \ge 1.
		\]
		Therefore, by the local freeness of $f_*'\mathcal{O}_{X'}(mK_{X'})$, we have
		\[
		P_m(X_t') = \mathrm{rk}\, f_*'\mathcal{O}_{X'}(mK_{X'})(t)
		= \mathrm{rk}\, f_*'\mathcal{O}_{X'}(mK_{X'})(0)
		= P_m(X_0'),
		\qquad m \ge 1.
		\]
		It follows that $P_m(X_t)$ is also constant in the family for all $m \ge 1$.

		\noindent\textbf{Case 2:} The relative $K_X$-MMP terminates at a relative Mori fiber space $\psi : X' \to Z/S$, hence
		\[
		h^0\!\left(X_t',\, \mathcal{O}_{X_t'}(mK_{X'})\right) = 0
		\qquad \text{for all } m \ge 1.
		\]
		Thus $P_m(X_t')=0$ for all $m\geq1$. Since $X_t$ has canonical
		singularities and is bimeromorphic to $X_t'$ for $t$ near $0$, we have
		$P_m(X_t)=P_m(X_t')=0$ for every $m\geq1$ and every $t$ in a
		neighborhood of $0$.
	\end{proof}
	\begin{remark}
		Chen--Rao--Tsai~\cite[Theorem~4.7]{CRT25} established a
		criterion for fiberwise bimeromorphism under the assumption that
		$\kappa(X_0)\geq0$. In the MMP setting considered above, our argument gives
		a complementary statement: if fiberwise bimeromorphism fails over every
		neighborhood of $0$, then $P_m(X_0)=0$ for all $m\geq1$, and hence
		$\kappa(X_0)=-\infty$.
	\end{remark}
	\subsection{Invariance of plurigenera for varieties of maximal Albanese dimension}
	In this subsection, we study the invariance of plurigenera in smooth families of varieties of maximal Albanese dimension. It
	extends Chen--Hacon's deformation invariance of Kodaira dimension
	\cite[Corollary~2]{CH01} to the invariance of all plurigenera.
	\begin{theorem}\label{t-maxAlbdim}
		Let $f:X\to S$ be a smooth K\"ahler family over a relatively compact,
		smooth, connected curve, and assume that $X_0$ has maximal Albanese
		dimension. Then, for every $m\geq1$, $P_m(X_t)$ is independent of $t$
		near $0$.
	\end{theorem}
	
	Let $(X,\omega)$ be a compact K\"ahler manifold. Integration defines the period map
	\[
	\iota_X:H_1(X,\mathbb Z)
	\longrightarrow H^0(X,\Omega_X^1)^\vee,
	\qquad
	\iota_X([\gamma])(\omega):=\int_\gamma\omega.
	\]
	The \emph{Albanese torus} of $X$ is the complex torus
	\[
	\operatorname{Alb}(X):=
	H^0(X,\Omega_X^1)^\vee
	\big/\iota_X\bigl(H_1(X,\mathbb Z)\bigr).
	\]
	After fixing a base point $x_0\in X$, the \emph{Albanese map} is defined by
	\[
	\alpha_X:X\longrightarrow\operatorname{Alb}(X),
	\qquad
	\alpha_X(x):=
	\left[\omega\longmapsto\int_{x_0}^{x}\omega\right].
	\]
	A different choice of $x_0$ changes $\alpha_X$ only by a translation. Recall
	that $X$ is said to have
	\emph{maximal Albanese dimension} if $\alpha_X$ is generically finite onto
	its image; equivalently,
	$\dim\alpha_X(X)=\dim X$.
	
	Varieties with maximal Albanese dimension have the following characterization due to
	Fujino.
	\begin{proposition}[{\cite[Remark~2.2]{Fujinoalb}}]
		\label{p-maximal-albanese-cotangent}
		A compact K\"ahler manifold $X$ has maximal Albanese dimension if and
		only if its cotangent bundle $\Omega_X^1$ is generically generated by
		global sections; equivalently, the evaluation morphism
		\[
		H^0(X,\Omega_X^1)\otimes\mathcal O_X
		\longrightarrow\Omega_X^1
		\]
		is surjective at the generic point of $X$.
	\end{proposition}

	\begin{lemma}\label{l-Albdef}
		Let $f: X \to S$ be a smooth K\"ahler family over $S$. If the central fiber $X_0$ has maximal Albanese dimension, then so does every nearby fiber $X_t$.
	\end{lemma}
	
	\begin{proof}
		By Ehresmann's theorem, $f:X\to S$ is locally trivial as a smooth fiber
		bundle and their first
		Betti number
		\[
		b_1(X_t):=\dim_{\mathbb Q}H^1(X_t,\mathbb Q)
		\]
		is locally constant in $t$. Since every fiber $X_t$ is a compact K\"ahler
		manifold, Hodge decomposition and Hodge symmetry give
		\[
		H^1(X_t,\mathbb C)
		\cong H^{1,0}(X_t)\oplus H^{0,1}(X_t)
		\cong H^0(X_t,\Omega^1_{X_t})\oplus H^1(X_t,\mathcal O_{X_t})
		\]
		and
		\[
		\dim H^0(X_t,\Omega^1_{X_t})
		=\dim H^1(X_t,\mathcal O_{X_t}).
		\]
		Consequently,
		\[
		q(X_t)=\dim H^1(X_t,\mathcal O_{X_t})=\frac{1}{2}b_1(X_t),
		\]
		so the irregularity is locally constant on $S$.
		
		We now prove that having maximal Albanese dimension is stable under deformation. Let
		$n=\dim X_t$ and let $q$ denote the
		common value of $q(X_t)$. Since $\dim H^0(X_t, \Omega_{X_t})$ is constant in $t$, by Grauert's theorem,
		\[
		\mathcal E:=f_*\Omega^1_{X/S}
		\]
		is locally free of rank $q$ on $S$, and commutes with base
		change. Consider the natural evaluation morphism
		\[
		\operatorname{ev}:f^*\mathcal E\longrightarrow\Omega^1_{X/S}.
		\]
		By \cite[Proposition~3.10]{HLR26}, its restriction to $X_t$ is
		\[
		\operatorname{ev}_t:
		H^0(X_t,\Omega^1_{X_t})\otimes\mathcal O_{X_t}
		\longrightarrow\Omega^1_{X_t}.
		\]
		Let
		\[
		W:=X\setminus\operatorname{Supp}\bigl(\operatorname{coker}(\operatorname{ev})\bigr).
		\]
		This is a Zariski open subset of $X$. Since $X_0$ has maximal Albanese dimension, Proposition~\ref{p-maximal-albanese-cotangent} implies that
		$\operatorname{ev}_0$ is generically surjective; hence
		$W\cap X_0\neq\varnothing$. The smooth morphism $f$ is open, so $f(W)$ is an
		open neighborhood of $0$ in $S$. After shrinking $S$ around $0$, we may
		therefore assume that $W\cap X_t\neq\varnothing$ for every $t\in S$. Thus
		$\operatorname{ev}_t$ is generically surjective on $X_t$.
		Therefore, by Proposition~\ref{p-maximal-albanese-cotangent}, every nearby fiber $X_t$ has maximal Albanese dimension.
	\end{proof}

	\begin{lemma}[{\cite[Lemma~3.2]{Fujinoalb}}]\label{l-refnefalb}
		Let $(X,B)$ be a klt K\"ahler pair and $f:X\to S$ a surjective morphism
		of compact complex analytic varieties. Assume that $K_X+B$ is $f$-nef and
		$S$ contains no rational curves. Then $K_X+B$ is nef.
	\end{lemma}
	\begin{proof}
		Suppose that $K_X+B$ is not nef. By the cone theorem
		(\cite[Theorem~0.5]{HP24} and also \cite[Corollary 3.6]{HLX26}), there exists a rational curve $C$ on $X$ such
		that $\left(K_X+B\right)\cdot C<0$. Since $K_X+B$ is $f$-nef, the image
		$f(C)$ is not a point. This contradicts the assumption that $S$ contains
		no rational curves. Therefore, $K_X+B$ is nef.
	\end{proof}
	The following existence result for the minimal model program will be used in our
	proof.
	\begin{lemma}[{\cite[Proposition~2.27]{HLX26}}]\label{l-existmmp}
		Let $f:X\to Y$ be a contraction morphism of normal compact complex
		varieties. Assume that $(X,B+\boldsymbol{\beta})$ is a generalized klt pair,
		that $B+\boldsymbol{\beta}_X$ is modified big over $Y$, and that
		$\boldsymbol{\beta}\equiv_Y\mathbf N$, where $\mathbf N$ is an
		$\mathbb R$-Cartier $b$-divisor.
		\begin{enumerate}
			\item[(1)] If $K_X+B+\boldsymbol{\beta}_X$ is pseudoeffective over $Y$,
			then $(X,B+\boldsymbol{\beta})$ has a good minimal model over $Y$, not
			necessarily $\mathbb Q$-factorial.
			\item[(2)] If $K_X+B+\boldsymbol{\beta}_X$ is not pseudoeffective over
			$Y$, then $(X,B+\boldsymbol{\beta})$ is birational to a Fano fibration
			over $Y$.
		\end{enumerate}
	\end{lemma}
	
	The following theorem is a relative version of
	\cite[Theorem~4.3]{Fujinoalb} for smooth K\"ahler families of varieties of
	maximal Albanese dimension.
	\begin{theorem}\label{Existencegoodmm}
		Let $f:X\to S$ be a smooth K\"ahler family over a relative compact, smooth connected curve $S$ such that $X_0$ has maximal Albanese dimension. Then, after shrinking $S$ around $0$, the family admits a relative good minimal model.
	\end{theorem}
	\begin{proof}
		By Lemma~\ref{l-Albdef}, maximal Albanese dimension is preserved under small deformations. Hence, after shrinking $S$ around $0$, every fiber $X_t$ has maximal Albanese dimension. Let
		\[
		\alpha : X \longrightarrow \mathcal A := \operatorname{Alb}(X/S)
		\]
		be the relative Albanese map, which is generically finite, and write
		$\alpha = g\circ\beta$ for its Stein factorization, with $\beta : X \to \mathcal A^{\mathrm{st}}$
		bimeromorphic and $g$ finite. Since $g$ is finite, a class is nef, semiample, or big
		over $\mathcal A^{\mathrm{st}}$ if and only if it is so over $\mathcal A$, so we may work over
		$\mathcal A^{\mathrm{st}}$ throughout. By Lemma \ref{l-existmmp}, a $K_X$-good minimal
		model over $\mathcal A$ exists,
		\begin{center}
			\begin{tikzcd}
				X && Y \\
				& {\mathcal{A}} \\
				& S
				\arrow[dashed, from=1-1, to=1-3]
				\arrow["\alpha", from=1-1, to=2-2]
				\arrow["f"', from=1-1, to=3-2]
				\arrow["{\alpha'}"', from=1-3, to=2-2]
				\arrow["{f'}", from=1-3, to=3-2]
				\arrow[from=2-2, to=3-2]
			\end{tikzcd}
		\end{center}
		where $Y$ is strongly $\mathbb Q$-factorial and klt, and $K_Y$ is
		semiample over $\mathcal A$, with induced morphisms $\alpha':Y\to\mathcal A$ and $f':Y\to S$.
		
		We claim that $K_Y$ is semi-ample over $S$. Since $K_Y$ is nef over
		$\mathcal A$ and the fibers $\mathcal A_t$ are complex tori containing no rational curves,
		Lemma \ref{l-refnefalb} shows that $K_Y$ is nef over $S$. Hence for general $t\in S$ the induced
		map $X_t\dashrightarrow Y_t$ is a minimal model of $X_t$. By \cite[Theorem~1.1]{DHAlbanese}, the fibers $X_t$
		admit good minimal models, so $K_{Y_t}$ is semiample for general $t$. By the base point free theorem
		\cite[Theorem~5.8]{Naka87}, $K_Y$ is semiample over $S$, so that $X\dashrightarrow Y$ is a
		relative good minimal model of $X$ over $S$.
	\end{proof}
	As a consequence, we obtain the invariance of plurigenera for smooth families of varieties of maximal Albanese dimension.
	
	\begin{proof}[Proof of Theorem \ref{t-maxAlbdim}]
		The proof is very similar to that of Theorem~\ref{t-threefold-plurigenera}. By Theorem~\ref{Existencegoodmm}, we may run a relative $K_X$-MMP over $S$,
		\[
		X\dashrightarrow X'/S,
		\]
		which terminates with a relative good minimal model such that $K_{X'}$ is semiample over $S$. By Proposition~\ref{p-kollar-torsion freeness}, $R^i f'_*\mathcal{O}_{X'}(mK_{X'})$ is locally free for every $i\geq 0$ and $m\ge 1$. Therefore, the plurigenera $P_m(X_t)$ are constant in the family.
	\end{proof}
	
	\subsection{Invariance of log plurigenera}
	
	We next turn to log plurigenera. Using Theorem~\ref{ExtMMP}, we obtain the
	following result, which slightly generalizes Koll\'ar's theorem
	\cite{DefGeneralType} to the log setting.
	
	\begin{theorem}\label{t-log-plurigenera}
		Let $g:X\to S$ be a smooth family such that $(X,B)$ is log smooth over a relatively compact, smooth connected curve $S$. Suppose that the special fiber $(X_0,B_0)$ is projective with canonical singularities, that $\lfloor B\rfloor=0$, and that $K_{X_0}+B_0$ is big.
		Assume moreover that the negative part of its Boucksom--Zariski
		decomposition satisfies
		$$
		N(K_{X_0}+B_0)\wedge B_0 = 0 .
		$$
		Then there exists an open neighborhood $U\subset S$ of $0$ such that,
		for every integer $m\geq1$ such that $mB$ is integral, the $m$-th log plurigenus
		$$
		P_m(X_t,B_t)=h^{0}\!\left(X_t,\, m(K_{X_t}+B_t)\right)
		$$
		is independent of $t\in U$.
	\end{theorem}
	\begin{proof}
		Our proof is motivated by \cite[Theorem~1]{DefGeneralType}. Since
		$(X_0,B_0)$ is projective, \cite{BCHM} allows us to run the
		$(K_{X_0}+B_0)$-MMP with scaling, which terminates with a minimal model:
		\[
		(X_0,B_0)\dashrightarrow(X_0^{(1)},B_0^{(1)})\dashrightarrow\cdots
		\dashrightarrow(X_0',B_0').
		\]
		By Theorem~\ref{ExtMMP}, this MMP extends over a connected neighborhood
		$U\subset S$ to a sequence of $(K_X+B)$-negative bimeromorphic contractions $$(X,B)\dashrightarrow(X^{(1)},B^{(1)})\dashrightarrow\cdots
		\dashrightarrow(X',B')/U,$$denote $\pi' : X' \to U$. Since
		$K_{X_0'}+B_0'$ is nef and big and $(X_0',B_0')$ is klt, so fix $i>0$
		and $m\ge 2$ such that $mB$ is integral. By Kawamata--Viehweg
		vanishing,
		$H^i\!\left(X_0',
		\mathcal{O}_{X_0'}\bigl(m(K_{X_0'}+B_0')\bigr)\right)=0$. Then the
		upper semicontinuity of cohomological dimension gives $$
		H^i\!\left(X_t',
		\mathcal O_{X_t'}\bigl(m(K_{X_t'}+B_t')\bigr)\right)=0,$$
		for every nearby $t$. In particular, $h^1(X_t',\mathcal O(m(K_{X_t'}+B_t')))=0$ for $t$ near $0$. Since cohomology in degree $-1$ is also zero,
		Corollary~\ref{l-squeeze}, applied with $q=0$, shows that
		\[
		P_m(X_t',B_t')=h^0\!\left(X_t',
		\mathcal O_{X_t'}\bigl(m(K_{X_t'}+B_t')\bigr)\right)
		\]
		is locally constant in $t\in U_m$ for some open neighborhood $0\in U_m$, and the degree $0$ base change map $\varphi^{0}_m (0):\pi_*'\mathcal O_{X'}\bigl(m(K_{X'}+B')\bigr)(0)\to H^0\!\left(X_0',
		\mathcal O_{X_0'}\bigl(m(K_{X_0'}+B_0')\bigr)\right)$ is an isomorphism for all $m\ge 2$ with $mB$ integral.
		
		A priori the neighborhood $U_m$ could shrink to $\{0\}$ as $m$ grows.
		We show that the same neighborhood works for all $m$. Set
		$L=\mathcal{O}_{X'}(N(K_{X'}+B'))$, where $N\ge 2$ is chosen so that
		$N(K_{X'}+B')$ is Cartier. By our discussion,
		$\varphi_{\ell N}^{(0)}(0)$ is an isomorphism for every $\ell\ge 1$;
		we thus have the isomorphism
		$$\bigoplus_{\ell\ge0}(\pi_*'L^\ell)\otimes k(0)\;\cong\;
		\bigoplus_{\ell\ge0}H^0\bigl(X'_0,\,
		\ell N(K_{X'_0}+B'_0)\bigr),$$
		where $L^\ell:=\mathcal O_{X'}(\ell N(K_{X'}+B'))$. On the other hand,
		since $K_{X_0'}+B_0'$ is semi-ample, the right-hand side is finitely
		generated. Then, by \cite[Lemma 11]{DefGeneralType}, there exists a
		neighborhood $0\in U\subset S$, independent of $m$, on which
		$\bigoplus_\ell\pi_*'L^\ell$ is a finitely generated
		$\mathcal O_U$-algebra. We then take
		$$Z:=\operatorname{Proj}_U\bigoplus_{\ell\ge0}\pi_*'L^\ell
		\longrightarrow U,$$ with $K_Z+B_Z$ ample over $U$ and an associated
		contraction $\varphi:X'\to Z$ such that
		$\varphi^*(K_Z+B_Z)=K_{X'}+B'$. For every $m\ge2$ such that $mB$ is
		integral, Kawamata--Viehweg vanishing gives
		\[
		H^i\bigl(Z_t,\mathcal O_{Z_t}(m(K_{Z_t}+B_{Z_t}))\bigr)=0
		\qquad (i>0,\ t\in U),
		\]
		where $U$ is independent of $m$. Since the cohomology in degree $-1$
		also vanishes, Corollary~\ref{l-squeeze}, applied with $q=0$, shows
		that $P_m(Z_t,B_{Z_t})$ is constant on $U$. Hence $P_m(X_t',B_t')$ is constant on $U$ as
		well.
		
		By \cite[Lemma~2.22]{CRT25}, after
		shrinking $U$ if necessary, the birational map $X\dashrightarrow X'$
		is fiberwise bimeromorphic; in particular, $X_t$ and $X_t'$ are
		bimeromorphic for every $t$ near $0$. Since the MMP is $(K_X+B)$-negative, the log plurigenera are preserved at each
		step. Consequently,
		\[
		P_m(X_t,B_t)=\dim H^0\!\left(X_t,
		\mathcal O_{X_t}\bigl(m(K_{X_t}+B_t)\bigr)\right)
		\]
		is locally constant in $t$ near $0$.
		
		It remains to consider the case $m=1$. In this case, by assumption $B$ is integral; since $\lfloor B\rfloor=0$, we have $B=0$. Serre duality gives
		\[
		h^0\!\left(X_t,\mathcal{O}_{X_t}(K_{X_t})\right)
		=h^n\!\left(X_t,\mathcal{O}_{X_t}\right), 
		\]where $n = \dim X_t$.
		The result then follows from \cite[Complement 2.62.5]{Kolbook}.
	\end{proof}

	\subsection{Invariance of higher cohomological plurigenera}

	We now turn to the deformation invariance of higher cohomological
	plurigenera. We collect the extension and torsion-freeness results we need in the following theorem. Part~(1) combines the injectivity theorems of
	Cao--Demailly--Matsumura~\cite[Theorem~1.1]{CDM17} and
	Zhou--Zhu~\cite[Theorem~1.1]{ZZ22}; part~(2) is Matsumura's relative
	torsion freeness and vanishing theorem \cite[Corollary~1.5]{Matsumura22};
	and part~(3) is a consequence of (1).
	
	\begin{theorem}[{\cite{CDM17,ZZ22,Matsumura22}}]
		\label{Exthigherdeg}
		\leavevmode
		Let $X$ be a complex manifold and $(L,h)$ a holomorphic line
		bundle on $X$ endowed with a possibly singular Hermitian metric.
		\begin{itemize}
			\item[(1)]\phantomsection\label{Exthigherdeg-part-1} Let $\psi$ be a quasi-plurisubharmonic function with neat
			analytic singularities and assume that $X$ is holomorphically convex and
			K\"ahler. Suppose that there is a positive continuous function
			$\delta$ on $X$ such that
			\[
			\sqrt{-1}\Theta_{L,h}
			+(1+\alpha\delta)\sqrt{-1}\partial\bar\partial\psi\geq0
			\qquad\text{for every }\alpha\in[0,1].
			\]
			Then the natural inclusion
			$\mathcal I(he^{-\psi})\hookrightarrow\mathcal I(h)$ induces an
			injective map
			\[
			H^q\!\left(X,\mathcal O_X(K_X+L)\otimes\mathcal I(he^{-\psi})\right)
			\longrightarrow
			H^q\!\left(X,\mathcal O_X(K_X+L)\otimes\mathcal I(h)\right)
			\]
			for every $q\geq0$. Equivalently, for every $q\geq0$, the quotient
			map induces a surjection
			\[
			H^q\!\left(X,\mathcal O_X(K_X+L)\otimes\mathcal I(h)\right)
			\longrightarrow
			H^q\!\left(X,\mathcal O_X(K_X+L)\otimes
			\frac{\mathcal I(h)}{\mathcal I(he^{-\psi})}\right).
			\]
			
			\item[(2)]\phantomsection\label{Exthigherdeg-part-2} Suppose that $f:X\to S$ is a surjective proper K\"ahler
			morphism to an analytic space and that $h$ has semipositive curvature.
			Then
			\[
			R^qf_*\!\left(\mathcal O_X(K_X+L)\otimes\mathcal I(h)\right)
			\]
			is torsion free for every $q\geq0$, and it vanishes for
			$q>\dim X-\dim S$.
			
			\item[(3)]\phantomsection\label{Exthigherdeg-part-3} Assume, in addition to the hypotheses of
			part~\hyperref[Exthigherdeg-part-2]{(2)}, that
			$S$ is a smooth curve, $f$ is smooth, and $h$ is smooth. For any
			$0\in S$, after shrinking $S$ to a coordinate neighborhood of $0$,
			the restriction map
			\[
			H^q\!\left(X,\mathcal O_X(K_X+L)\right)\longrightarrow
			H^q\!\left(X_0,\mathcal O_{X_0}(K_X+L)\right)
			\]
			is surjective for every $q\geq0$.
		\end{itemize}
	\end{theorem}
	\begin{proof}
		Part~\hyperref[Exthigherdeg-part-1]{(1)} is the injectivity theorem of
		\cite[Theorem~1.1]{CDM17} and
		\cite[Theorem~1.1]{ZZ22}. Indeed, tensoring
		the short exact sequence
		\[
		0\longrightarrow\mathcal I(he^{-\psi})
		\longrightarrow\mathcal I(h)
		\longrightarrow
		\frac{\mathcal I(h)}{\mathcal I(he^{-\psi})}
		\longrightarrow0
		\]
		with $\mathcal O_X(K_X+L)$ and taking cohomology gives a connecting map from
		the quotient cohomology in degree $q$ to
		$H^{q+1}\!\left(X,\mathcal O_X(K_X+L)\otimes
		\mathcal I(he^{-\psi})\right)$. The injectivity
		assertion in degree $q+1$ forces this connecting map to vanish, and hence
		the quotient map in degree $q$ is surjective.
		Part~\hyperref[Exthigherdeg-part-2]{(2)} is
		\cite[Corollary~1.5]{Matsumura22}.
		
		For part~\hyperref[Exthigherdeg-part-3]{(3)}, let $t$ be a local coordinate on $S$ centered at $0$.
		Since $f$ is proper and the coordinate neighborhood is Stein, its
		preimage is holomorphically convex. Set $s=f^*t$ and take
		$\psi=\log|s|^2$. Since $h$ is smooth,
		\[
		\mathcal I(h)=\mathcal O_X,
		\qquad
		\mathcal I(he^{-\psi})=\mathcal I_{X_0}
		\]
		by \cite[Remark~5.9]{Dem12}. The Poincar\'e--Lelong formula gives
		$\sqrt{-1}\partial\bar\partial\psi=[X_0]\geq0$, so the curvature
		hypothesis in part~\hyperref[Exthigherdeg-part-1]{(1)} is satisfied. Its surjective formulation now
		gives precisely the asserted restriction map.
	\end{proof}
	This gives the following invariance of higher cohomological
	plurigenera in smooth K\"ahler families when $K_X$ is semipositive.
	\begin{corollary}\label{c-higher-plurigenera}
		Let $f:X\to S$ be a smooth family with $X$ K\"ahler. If $K_X$ is semipositive, then, for every $m\geq 1$ and $q\geq 0$, the higher cohomological plurigenus
		\[
		P_m^{(q)}(X_t)=\dim H^q(X_t,mK_{X_t})
		\]
		is independent of $t\in S$.
	\end{corollary}
	\begin{proof}
		Work near an arbitrary point $0\in S$. Applying Theorem~\ref{Exthigherdeg} to $L=(m-1)K_X$, equipped with the induced semipositive metric, shows that the restriction maps in every degree are surjective. Theorem~\ref{l-Extension} therefore gives local constancy, and connectedness of $S$ gives constancy on the entire base.
	\end{proof}

	\subsection{Invariance of generalized plurigenera in Moishezon families}
	We next study deformation invariance for generalized pairs in the Moishezon
	setting. For projective families,
	Filipazzi--Svaldi~\cite[Theorem~4.1]{FS20} established deformation
	invariance of plurigenera for big and nef klt generalized pairs as part of
	their study of boundedness. Li--Wang~\cite{LW26} subsequently proved an
	invariance theorem for generalized polarized pairs with abundant nef
	parts. We extend the latter argument to proper Moishezon
	morphisms.
	
	\begin{theorem}\label{t-ginvarpluri}
		Let $f:X\to S$ be a proper Moishezon morphism with connected fibers over a relatively compact, smooth connected curve $S$ and
		$(X,\mathbf M)$ a generalized pair with abundant nef part over $S$. Write
		\[
		K_{X_t}+\mathbf M_t
		=\left.(K_X+\mathbf M_X)\right|_{X_t},
		\]
		and assume that $(X_t,\mathbf M_t)$ has generalized canonical
		singularities for every $t\in S$. Then, for every $m\geq1$ such that
		$m\mathbf M_X$ is Cartier, the generalized $m$-th plurigenera
		\[
		t\longmapsto
		\dim H^0\!\left(X_t,
		\mathcal O_{X_t}\bigl(m(K_{X_t}+\mathbf M_t)\bigr)\right)
		\]
		is constant on $S$.
	\end{theorem}
	
	We first collect some preliminary results needed for the proof.
	
	\begin{definition}[Relatively nef and abundant divisor]\label{d-abundant}
		Let $p:X\to S$ be a proper surjective morphism whose general fiber lies
		in Fujiki's class $\mathcal C$. A relatively nef $\mathbb Q$-Cartier
		divisor $L$ is \emph{abundant} (or \emph{good}) over $S$ if
		$\nu(L/S)=\kappa(L/S)$, where $\kappa(\cdot)$ is the Kodaira dimension and $\nu(\cdot)$ is the numerical Kodaira dimension. Recall that, for a nef $\mathbb Q$-Cartier
		divisor $A$ on a compact complex variety, its \emph{numerical Kodaira dimension}
		is defined by
		\[
		\nu(A):=\max\left\{\ell\geq0\;\middle|\; A^\ell\not\equiv0\right\}.
		\]
		For an analytically sufficiently general point $t\in S$, set
		\[
		\nu(L/S):=\nu\!\left(L|_{X_t}\right),
		\]
		where $X_t=p^{-1}(t)$.
	\end{definition}

	The structure theorem below for nef and abundant divisors is useful:
	
	\begin{proposition}[{\cite[Proposition~2.17]{Naka87}}]
		\label{p-nefabundant}
		Let $f:X\to S$ be a proper Moishezon morphism, and let $L$ be a relatively nef and abundant
		$\mathbb Q$-Cartier divisor on $X$ (Definition~\ref{d-abundant}). For every
		$s\in S$, there exists an open neighborhood $U\subset S$ of $s$ such that,
		writing $X_U:=f^{-1}(U)$, there are a modification
		$\mu:X'\to X_U$, a proper contraction $g:X'\to Z$, and a projective
		morphism $h:Z\to U$ fitting into a commutative diagram
		\[
		\begin{tikzcd}
			X' \arrow[r,"\mu"] \arrow[d,"g"'] & X_U \arrow[d,"f|_{X_U}"] \\
			Z \arrow[r,"h"'] & U,
		\end{tikzcd}
		\]
		such that
		\[
		\mu^*\!\left(L|_{X_U}\right)\sim_{\mathbb Q,U}g^*H
		\]
		for some $h$-nef and $h$-big $\mathbb Q$-Cartier divisor $H$ on $Z$.
	\end{proposition}
	Motivated by the preceding structure theorem, we introduce the notion of a
	generalized pair with abundant nef part.
	\begin{definition}[Generalized pair with abundant nef part]\label{def-nefabun}
		A $\mathbb Q$-divisorial generalized pair $(X/S,B+\mathbf M)$, with
		structural morphism $f:X\to S$, has an \emph{abundant nef part} if, for
		every $s\in S$, there exists an open neighborhood $U\subset S$ of $s$
		such that, writing $X_U:=f^{-1}(U)$, there are a modification
		$\mu:X'\to X_U$ to which $\mathbf M|_{X_U}$ descends and a factorization
		of $f|_{X_U}\circ\mu$
		\[
		X'\xrightarrow{g}Z\xrightarrow{h}U
		\]
		such that
		\[
		\mathbf M_{X'}\sim_{\mathbb Q,U}g^*H
		\]
		for some $h$-nef and $h$-big $\mathbb Q$-Cartier divisor $H$ on $Z$, and $h:Z\to U$ is projective.
	\end{definition}
	
	Thus, locally over $S$ and after passing to a higher model, the nef part
	factors through a projective morphism $h:Z\to U$. This is the essential replacement for the
	projectivity assumption in \cite{LW26}. Under this assumption, we can now construct, locally over $S$, a special metric $h_k$ on the trace $\mathbf M_X$ of the b-divisor $\mathbf M$.
	
	We first fix the notation used below. Let $L$ be a
	$\mathbb Q$-Cartier divisor, choose $m>0$ such that $mL$ is Cartier, and
	let $s_i\in H^0(X,\mathcal O_X(mL))$. For a local trivialization
	$\theta:\mathcal O_X(mL)|_U\simeq\mathcal O_U$, set
	$f_i:=\theta(s_i)$. The metric defined by the sections $s_i$ has local
	weight
	\[
	\varphi:=\frac{1}{2m}\log\!\left(\sum_i|f_i|^2\right).
	\]
	For an effective $\mathbb Q$-Cartier divisor $A$, the canonical
	divisorial metric $\overline h_A$ is the special case defined by the
	canonical section of $qA$ for a sufficiently divisible $q$; its local
	representative is a defining equation $f_{qA}$ of $qA$, and hence
	$\varphi_A=q^{-1}\log|f_{qA}|$. If $\mu:X'\to X$ is a modification and
	$h'$ is defined by sections $\sigma_i$ of $mD$, the push-forward metric
	$\mu_*h'$ is defined by the same construction using the sections
	$\mu_*\sigma_i$ of $m\mu_*D$.
	
	With these metric conventions, we adapt the constructions of
	\cite[Lemmas~4.2--4.4]{LW26} to the relative Moishezon setting in the
	following lemma.
	
	\begin{lemma}\label{l-bignef}
		In the setting of Definition~\ref{def-nefabun}, fix $0\in S$. After
		shrinking $S$ around $0$, choose $\mu:X'\to X$ and a factorization
		$X'\xrightarrow{g}Z\xrightarrow{h}S$ as in that definition, and denote by
		$\widetilde X_0$ the strict transform of $X_0$ on $X'$. After replacing
		$X'$ by a higher model and shrinking $S$ further, there are effective
		$\mathbb Q$-divisors $E,F$ on $X'$ without common components, with $E$
		$\mu$-exceptional and $\widetilde X_0\not\subset\operatorname{Supp}(E+F)$,
		such that, for every $k\geq1$, there is an $h$-ample
		$\mathbb Q$-Cartier divisor $H_k$ satisfying
		\[
		\mathbf M_{X'}+\frac1kE
		\sim_{\mathbb Q,S}g^*H_k+\frac1kF.
		\]
		Moreover, this equivalence defines a metric $h_k'$ by global sections on
		$\mathbf M_{X'}+\frac1kE$ and a push-forward metric
		$h_k:=\mu_*h_k'$ on $\mathbf M_X$, both with nonnegative curvature
		current and with nontrivial restriction to the central fiber (we write $h_{k,0}:=h_k|_{X_0}$).
		If
		\[
		\mu^*\mathbf M_X=\mathbf M_{X'}+\Xi
		\]
		and $\varphi_k,\varphi_E,\varphi_F,\varphi_\Xi$ denote the corresponding
		local weights, let
		$\varphi_{k,0},\varphi_{E_0},\varphi_{F_0},\varphi_{\Xi_0}$ denote their
		restrictions to the central fiber. Then
		\[
		\mu^*\varphi_k+\frac1k\varphi_E
		\approx\frac1k\varphi_F+\varphi_\Xi,
		\qquad
		\mu_0^*\varphi_{k,0}+\frac1k\varphi_{E_0}
		\approx\frac1k\varphi_{F_0}+\varphi_{\Xi_0}.
		\]
	\end{lemma}
	
	\begin{proof}
		By Proposition~\ref{p-nefabundant}, the intermediate morphism
		$h:Z\to S$ is projective, and
		$\mathbf M_{X'}\sim_{\mathbb Q,S}g^*H$ for an $h$-nef and $h$-big
		divisor $H$. By the relative Kodaira lemma, there is
		an effective $\mathbb Q$-divisor $G$ on $Z$ such that, for every
		$k\geq1$,
		\[
		H\sim_{\mathbb Q,S}H_k+\frac1kG
		\]
		with $H_k$ $h$-ample. Set
		\[
		c:=\operatorname{mult}_{X_0}\!\left(\mu_*g^*G\right),\qquad
		D:=g^*G-c(f\circ\mu)^*(0).
		\]
		The divisor $D$ does not contain $\widetilde X_0$, and every component of
		its negative part is $\mu$-exceptional. As in
		\cite[Lemma~4.2]{LW26}, write $D=F-E$ as the difference of its positive and
		negative parts. Then $E$ and $F$ are effective divisors without common
		components, $E$ is $\mu$-exceptional, and
		$\widetilde X_0\not\subset\operatorname{Supp}(E+F)$. Since
		$(f\circ\mu)^*(0)$ is trivial over $S$, we obtain
		\[
		\mathbf M_{X'}+\frac1kE
		\sim_{\mathbb Q,S}g^*H_k+\frac1kF.
		\]
		This is the construction in \cite[Lemma~4.2]{LW26}. 
		
		After shrinking $S$, choose a sufficiently divisible integer $r$ such that
		$rH_k$ is Cartier and its relative complete linear system defines an
		embedding
		\[
		\iota_k:Z\hookrightarrow\mathbb P^N_S.
		\]
		Choose relative global sections
		$\sigma_0,\ldots,\sigma_N\in H^0(Z,\mathcal O_Z(rH_k))$ with no common
		zero that define the embedding $\iota_k$. If
		$\theta:\mathcal O_Z(rH_k)|_U\simeq\mathcal O_U$ is
		a local trivialization, the pullback of the relative Fubini--Study metric has
		local weight
		\[
		\varphi_{rH_k}=\frac12\log\!\left(
		\sum_{j=0}^N|\theta(\sigma_j)|^2\right).
		\]
		The sections have no common zero, so this weight is smooth and
		plurisubharmonic. Taking the $r$-th root gives a smooth metric $h_{H_k}$ on
		$H_k$ with nonnegative curvature, and we set
		\[
		\chi_k:=g^*h_{H_k}.
		\]
		Thus $\chi_k$ is a smooth semipositive metric on $g^*H_k$ defined by global
		sections. Twisting it by the canonical metric $\bar{h}_F$ on $F$ gives the metric
		\[
		h_k':=\chi_k\,\overline h_F^{1/k}
		\]
		on $g^*H_k+\frac1kF$, and therefore, via the equivalence above, on
		$\mathbf M_{X'}+\frac1kE$. Because $\widetilde X_0$ is not a component
		of $F$, the restriction of $h_k'$ to $\widetilde X_0$ is not identically infinite,
		and
		\[
		\varphi_k'\approx\frac1k\varphi_F,\quad
		\varphi_{k,0}'\approx\frac1k\varphi_{F_0}.
		\]
		Finally, the negativity lemma gives
		$\mu^*\mathbf M_X=\mathbf M_{X'}+\Xi$ with $\Xi\geq0$
		$\mu$-exceptional. Since $h_k'$ is defined by global sections, its
		push-forward $h_k:=\mu_*h_k'$ is a well-defined metric on $\mathbf M_X$.
		To compare the two metrics, observe that $h_k'$ is a metric on
		$\mathbf M_{X'}+\frac1kE$, while
		\[
		\mu^*\mathbf M_X
		=\mathbf M_{X'}+\Xi
		=\left(\mathbf M_{X'}+\frac1kE\right)
		+\left(\Xi-\frac1kE\right).
		\]
		Hence the difference between the two underlying divisors is
		$\Xi-\frac1kE$. By the definition of the push-forward metric and the
		multiplicativity of divisorial metrics,
		\[
		\mu^*h_k
		=h_k'\,\overline h_\Xi\,\overline h_E^{-1/k}
		\]
		up to multiplication by a positive locally bounded function. Equivalently,
		\[
		\mu^*h_k\cdot\overline h_E^{1/k}
		=h_k'\,\overline h_\Xi
		\]
		up to the same type of factor. Passing to local weights, we obtain
		\[
		\mu^*\varphi_k+\frac1k\varphi_E
		\approx\varphi_k'+\varphi_\Xi
		\approx\frac1k\varphi_F+\varphi_\Xi.
		\]
		Restricting this identity to the central fiber similarly gives
		\[
		\mu_0^*\varphi_{k,0}+\frac1k\varphi_{E_0}
		\approx\varphi_{k,0}'+\varphi_{\Xi_0}
		\approx\frac1k\varphi_{F_0}+\varphi_{\Xi_0}.
		\]
		These are precisely the asserted weight identities; the curvature currents
		are nonnegative because all the metrics involved are defined by global
		sections, as in \cite[Lemmas~4.3 and~4.4]{LW26}. Thus the only new input
		in the Moishezon setting is the projectivity of the auxiliary morphism
		$Z\to S$ supplied by Proposition~\ref{p-nefabundant}; the remainder of
		the metric construction is unchanged.
	\end{proof}
	
	We recall the definition of the bimeromorphic multiplier ideal sheaf from
	\cite[Definition~5.4]{LW26}. Let $h=e^{-2\varphi}$ be a singular metric on
	a line bundle $L$. For $x\in X$,
	the stalk $\mathcal G_m(h)_x$ consists of the germs
	$\alpha\in\mathcal O_{X,x}$ for which there exist a resolution
	$\nu:W\to X$ and a neighborhood $U$ of $x$ on which $\alpha$ is represented,
	such that
	\[
	|\nu^*\alpha|^2|J(\nu)|^{2m}e^{-2\nu^*\varphi}
	\in L^1_{\mathrm{loc}}\bigl(\nu^{-1}(U)\bigr).
	\]
	Here $J(\nu)$ is the Jacobian factor of
	the resolution.
	The Jacobian factor makes it stable under bimeromorphic changes (cf. \cite[Lemma~5.12]{LW26}).
	
	The following proposition  extends \cite[Theorem~1.1]{LW26} from projective
	families to Moishezon families. 
	\begin{proposition}
		\label{t-LWext}
		Let $\pi:X\to S$ be a Moishezon family with connected fibers over a
		relatively compact, smooth, connected curve and $(L,h)$ a line bundle
		with a singular Hermitian metric of nonnegative curvature current, and
		assume that $X_0$ has canonical singularities and $h|_{X_0}$ is well
		defined. Then every section of
		\[
		H^0\!\left(
		X_0,\mathcal O_{X_0}(mK_{X_0}+L|_{X_0})
		\otimes\mathcal G_m(h|_{X_0})\right)
		\]
		extends over a neighborhood of $X_0$.
	\end{proposition}
	
	\begin{proof}
		After shrinking $S$ around $0$, we fix
		\[
		s\in H^0\!\left(X_0,
		\mathcal O_{X_0}(mK_{X_0}+L|_{X_0})
		\otimes\mathcal G_m(h|_{X_0})\right).
		\]
		By \cite[Lemma~5.12]{LW26}, there is a resolution
		$\nu_0:\widetilde X_0\to X_0$ such that membership in
		$\mathcal G_m(h|_{X_0})$ is equivalent to the ordinary multiplier ideal
		condition:
		\[
		\nu_0^*s\in H^0\!\left(\widetilde X_0,
		\mathcal O_{\widetilde X_0}
		\bigl(mK_{\widetilde X_0}+\nu_0^*(L|_{X_0})\bigr)
		\otimes\mathcal I\bigl(\nu_0^*(h|_{X_0})\bigr)\right).
		\]
		Then, by \cite[Lemma~5.13]{LW26}, there are smooth complex manifold $Y_0\subset Y$, a
		resolution $\mu:Y\to X$, and a morphism
		$\tau_0:Y_0\to\widetilde X_0$ such that
		\[
		\mu_0:=\mu|_{Y_0}=\nu_0\circ\tau_0.
		\]
		By \cite[Lemma~5.11]{LW26}, the multiplier ideal condition is preserved
		after passing to this higher resolution; hence
		\[
		\mu_0^*s\in H^0\!\left(Y_0,
		\mathcal O_{Y_0}\bigl(mK_{Y_0}+\mu_0^*(L|_{X_0})\bigr)
		\otimes\mathcal I\bigl(\mu_0^*(h|_{X_0})\bigr)\right).
		\]
		
		Because $X\to S$ is Moishezon, we may first take a projective modification
		and then choose $\mu$ projective over $S$. Therefore
		Takayama's extension theorem \cite[Theorem~3.1]{Taka06} (and also \cite[Theorem 5.1]{LW26} in the twisted setting) extends
		$\mu_0^*s$ from $Y_0$ to $Y$. After a
		further shrinking of $S$, the total space $X$ has canonical singularities
		\cite{KawaDeform99}. Put $L_0:=L|_{X_0}$. The restriction and pullback fit
		into the commutative diagram
		\[
		\begin{tikzcd}[column sep=large, row sep=large]
			H^0\!\left(Y,\mathcal O_Y(mK_Y+\mu^*L)\right)
			\arrow[r,"\mathrm{res}"]
			& H^0\!\left(Y_0,
			\mathcal O_{Y_0}(mK_Y|_{Y_0}+\mu_0^*L_0)\right) \\
			H^0\!\left(X,\mathcal O_X(mK_X+L)\right)
			\arrow[u,"\mu^*"']
			\arrow[r,"\mathrm{res}"']
			& H^0\!\left(X_0,
			\mathcal O_{X_0}(mK_{X_0}+L_0)\right)
			\arrow[u,hook].
		\end{tikzcd}
		\]
		The left vertical map is an isomorphism because $X$ has canonical
		singularities. To describe the right vertical map, write
		\[
		K_{Y_0}=\mu_0^*K_{X_0}+A_0,
		\]
		where $A_0\geq0$ is $\mu_0$-exceptional because $X_0$ has canonical
		singularities. Moreover, if
		$\mu^*X_0=Y_0+\Theta$ and $\Theta_0:=\Theta|_{Y_0}$, then adjunction gives
		\[
		K_Y|_{Y_0}\sim K_{Y_0}+\Theta_0,
		\qquad \Theta_0\geq0.
		\]
		The right vertical arrow is therefore the composition
		\[
		\begin{aligned}
			H^0\!\left(X_0,\mathcal O_{X_0}(mK_{X_0}+L_0)\right)
			&\xrightarrow{\ \mu_0^*\ }
			H^0\!\left(Y_0,
			\mathcal O_{Y_0}(m\mu_0^*K_{X_0}+\mu_0^*L_0)\right)\\
			&\hookrightarrow
			H^0\!\left(Y_0,
			\mathcal O_{Y_0}(mK_{Y_0}+\mu_0^*L_0)\right)\\
			&\hookrightarrow
			H^0\!\left(Y_0,
			\mathcal O_{Y_0}(mK_Y|_{Y_0}+\mu_0^*L_0)\right),
		\end{aligned}
		\]
		where first map is injective because $\mu_0$ is birational, and the next two maps are inclusions
		induced by the effective divisors $mA_0$ and $m\Theta_0$, respectively.
		Hence the extension of $\mu_0^*s$ on $Y$ corresponds to a section on $X$
		whose restriction is $s$.
	\end{proof}
	
	Using the metric $h_k$ constructed in Lemma \ref{l-bignef}, we can prove the following extension theorem for pluricanonical systems of generalized pairs.
	\begin{proposition}\label{t-gpairinvPluri}
		Let $\pi:X\to S$ be as in Theorem~\ref{t-ginvarpluri} and
		$(X,\mathbf M)$ the generalized pair there, and fix $0\in S$. For
		every $m\geq1$ such
		that $m\mathbf M_X$ is Cartier, after shrinking $S$ around $0$, the
		restriction map
		\[
		H^0\!\left(X,
		\mathcal O_X\bigl(m(K_X+\mathbf M_X)\bigr)\right)
		\longrightarrow
		H^0\!\left(X_0,
		\mathcal O_{X_0}\bigl(m(K_{X_0}+\mathbf M_0)\bigr)\right)
		\]
		is surjective.
	\end{proposition}
	
	\begin{proof}
		After shrinking $S$, we may assume that $X$ has
		canonical singularities \cite{KawaDeform99}. Choose a model
		$\mu:X'\to X$ on which $\mathbf M$ descends and for which
		Proposition~\ref{p-nefabundant} and Lemma~\ref{l-bignef} apply. Define the discrepancy divisors
		$A_0^{\mathrm{gen}}$ and $A_0$ by
		\[
		K_{\widetilde X_0}+\mathbf M_{X'}|_{\widetilde X_0}
		=\mu_0^*(K_{X_0}+\mathbf M_0)+A_0^{\mathrm{gen}},
		\qquad
		K_{\widetilde X_0}=\mu_0^*K_{X_0}+A_0.
		\]
		Generalized canonicity gives $A_0^{\mathrm{gen}}\geq0$. Moreover, since $X_0$ has canonical singularities, $A_0\geq0$. Write
		\[
		\mu^*\mathbf M_X=\mathbf M_{X'}+\Xi,
		\qquad \Xi\geq0,
		\]
		where $\Xi$ is $\mu$-exceptional, and set
		$\Xi_0:=\Xi|_{\widetilde X_0}$. Thus
		$\mu_0^*\mathbf M_0=\mathbf M_{X'}|_{\widetilde X_0}+\Xi_0$, and comparison
		of the two discrepancy formulas gives
		\[
		A_0^{\mathrm{gen}}=A_0-\Xi_0.
		\]
		
		Fix $m\ge 1$ and a section
		$s\in H^0(X_0,\mathcal{O}_{X_0}(m(K_{X_0}+\mathbf M_0)))$. To apply
		Proposition~\ref{t-LWext}, we must show that, for $k\gg1$,
		\[
		s\in H^0\!\left(X_0,
		\mathcal O_{X_0}\bigl(m(K_{X_0}+\mathbf M_0)\bigr)
		\otimes\mathcal G_m(h_{k,0}^m)\right).
		\]
		If $\alpha$ denotes the local coefficient of $s$ in a trivialization,
		then the definition of $\mathcal G_m$ and \cite[Lemma~5.12]{LW26} reduce
		this membership to the local integrability condition
		\[
		\int_{\mu_0^{-1}(U)}
		\|\mu_0^*\alpha\|^2_{\mu_0^*(h_{k,0}^m)}
		|J(\mu_0)|^{2m}\,dV_{\widetilde X_0}<\infty\tag{*}
		\]
		for every sufficiently small open set. Note that
		$\operatorname{div}(J(\mu_0))
		=K_{\widetilde X_0}-\mu_0^*K_{X_0}=A_0$. The weight identities in
		Lemma~\ref{l-bignef} give
		\[
		-m\mu_0^*\varphi_{k,0}+m\varphi_{A_0}
		\approx
		m\left(-\frac1k\varphi_{F_0}
		+\frac1k\varphi_{E_0}+\varphi_{A_0^{\mathrm{gen}}}\right).
		\]
		Since $A_0^{\mathrm{gen}}$ and $E_0$ are effective, Skoda's integrability
		lemma (cf. \cite[Lemma 5.6]{Dem12}) shows that the integral $(*)$ is finite once $k$ is chosen so that
		\[
		\frac{2m}{k}\,
		\nu\!\left(\Theta_{\overline h_{F_0}},y\right)<1
		\quad\text{for every }y\in\widetilde X_0.
		\]
		Consequently, every section $s$ as above vanishing along
		$\mathcal G_m(h_{k,0}^m)$ for $k\gg1$.
		
		We now equip $m \mathbf{M}_X$ with the metric $h_k^m$. Proposition~\ref{t-LWext} then extends every section from the
		central fiber, which completes the proof. 
	\end{proof}
	We can now complete the proof:
	\begin{proof}[Proof of Theorem~\ref{t-ginvarpluri}]
		Apply Proposition~\ref{t-gpairinvPluri} near an arbitrary point of $S$. Thus
		every generalized pluricanonical section on a fiber extends locally over
		the base. The extension criterion in Theorem~\ref{l-Extension} gives the
		local constancy of the corresponding $h^0$ (see the proof of \cite[Corollary~1.3]{LW26}), and connectedness of $S$
		gives constancy on the whole base.
	\end{proof}

	\appendix
	\section{Rao--Tsai's dimension-invariance approach to torsion freeness}\label{app-base change}
	This appendix gives a direct proof of the equivalence between constancy of fiberwise cohomology dimensions and torsion freeness of the corresponding higher direct images over a smooth curve. Using the exact sequence defined by multiplication by a local parameter, we recover the equivalence between conditions \hyperref[main-cond-iii]{(iii)} and \hyperref[main-cond-v]{(v)} in Theorem~\ref{Main}, generalizing the exact-sequence argument of \cite[Proposition~4.12]{RT1}. This argument also makes explicit how torsion obstructs base change at a special fiber.
	\begin{proposition}
		Let $f:X \to Y$ be a proper morphism of complex spaces where $Y$ is a reduced and smooth complex space of dimension one, $\mathscr{F}$ a coherent analytic sheaf on $X$ which is flat with respect to $f$, $y$ a point in $Y$, and $q$ an integer. Then $h^q(X_y,\mathscr{F}_{(y)})$ is locally constant if and only if $R^{q+1}f_*\mathscr{F}$ and $R^{q}f_*\mathscr{F}$ are torsion free.
	\end{proposition}
	\begin{proof}
		Fix a point $y\in Y$. Since the argument is local on $Y$, we may replace $Y$ by a coordinate neighborhood of $y$ and $X$ by its inverse image. Let $t$ be a local coordinate with $t(y)=0$, and denote its pullback to $X$ by the same symbol. Flatness of $\mathscr{F}$ over $Y$ implies that multiplication by $t$ on $\mathscr{F}$ is injective and hence defines an isomorphism $\mathscr{F}\xrightarrow{\sim}t\mathscr{F}$. The inclusion $\iota:t\mathscr{F}\hookrightarrow\mathscr{F}$ fits into the short exact sequence
		$$
		0 \to t\mathscr{F} \xrightarrow{\iota} \mathscr{F} \to \frac{\mathscr{F}}{t\mathscr{F}} \to 0.
		$$
		Applying the higher direct image functors gives the long exact sequence
		$$
		\cdots \to R^qf_*\mathscr{F} \xrightarrow{\gamma_q} R^qf_*(\frac{\mathscr{F}}{t\mathscr{F}}) \xrightarrow{\delta_q} R^{q+1}f_*(t\mathscr{F}) \xrightarrow{\iota_{q+1}} R^{q+1}f_*\mathscr{F} \to \cdots.
		$$
		Here $\iota_k:=R^kf_*(\iota)$. Let
		$\nu:t\mathscr{F}\xrightarrow{\sim}\mathscr{F}$ be the inverse of the
		multiplication isomorphism, given by $ts\mapsto s$. For each $k$, it induces
		an isomorphism
		\[
		\jmath_k:=R^kf_*(\nu):R^kf_*(t\mathscr{F})
		\xrightarrow{\sim}R^kf_*\mathscr{F}.
		\]
		Since $\iota=(\cdot t)\circ\nu$, functoriality gives
		\[
		\iota_k=(\cdot t)\circ\jmath_k.
		\]
		Thus $\iota_k$ factors through $\jmath_k$, followed by multiplication by
		$t$ on $R^kf_*\mathscr{F}$.
		
		For the ``if'' direction, assume that $R^{q+1}f_*\mathscr{F}$ and
		$R^qf_*\mathscr{F}$ are torsion free. The factorization
		\[
		\iota_{q+1}=(\cdot t)\circ\jmath_{q+1}
		\]
		shows that $\iota_{q+1}$ is injective, since $\jmath_{q+1}$ is an
		isomorphism and multiplication by $t$ on $R^{q+1}f_*\mathscr{F}$ is
		injective by the torsion freeness of $R^{q+1}f_*\mathscr{F}$. Exactness gives $\delta_q=0$, so $\gamma_q$ is surjective.
		The preceding map in the long exact sequence is $\iota_q$; since
		$\jmath_q$ is an isomorphism, its image is $t\cdot R^qf_*\mathscr{F}$. Hence
		\[
		\ker(\gamma_q)=t\cdot R^qf_*\mathscr{F},
		\]
		and we obtain an isomorphism
		\[
		\frac{R^qf_*\mathscr{F}}{t\cdot R^qf_*\mathscr{F}}
		\cong R^qf_*\bigl(\mathscr{F}/t\mathscr{F}\bigr).
		\]
		Since $\mathscr{F}/t\mathscr{F}$ is the direct image of
		$\mathscr{F}_{(y)}$ under the inclusion $X_y\hookrightarrow X$,
		Proposition \ref{higher stalk} yields
		\[
		\bigl(R^qf_*(\mathscr{F}/t\mathscr{F})\bigr)_y
		\cong H^q(X_y,\mathscr{F}_{(y)}).
		\]
		Taking stalks at $y$, we therefore obtain
		\[
		\frac{(R^qf_*\mathscr{F})_y}{t\cdot(R^qf_*\mathscr{F})_y}
		\cong H^q(X_y,\mathscr{F}_{(y)}).
		\]
		The left-hand side is the fiber of $R^qf_*\mathscr{F}$ at $y$.
		By Theorem \ref{direct image}, $R^qf_*\mathscr{F}$ is coherent;
		since it is torsion free on the smooth curve $Y$, it is locally free.
		As $y$ was arbitrary, the function
		$y\mapsto h^q(X_y,\mathscr{F}_{(y)})$ equals its locally constant rank
		function and is therefore locally constant.
		
		For the ``only if'' direction, assume that
		$z\mapsto h^q(X_z,\mathscr{F}_{(z)})$ is locally constant.
		By Grauert's theorem on cohomology and base change,
		$R^qf_*\mathscr{F}$ is locally free, hence torsion free, and the base
		change map at $y$ is an isomorphism. Under the identification above,
		this implies that $(\gamma_q)_y$ is surjective. Exactness then gives
		$(\delta_q)_y=0$, so $(\iota_{q+1})_y$ is injective.
		Suppose that $(R^{q+1}f_*\mathscr{F})_y$ has a nonzero torsion element.
		Since $\mathcal{O}_{Y,y}$ is a discrete valuation ring with uniformizer
		$t$, there exist a nonzero element $u$ and a least integer $n\geq1$
		such that $t^n\cdot u=0$. Setting $s=t^{n-1}\cdot u$, we have $s\neq0$ and $t\cdot s=0$.
		The isomorphism $(\jmath_{q+1})_y$ gives a nonzero element
		\[
		x=(\jmath_{q+1})_y^{-1}(s)
		\in\bigl(R^{q+1}f_*(t\mathscr{F})\bigr)_y.
		\]
		The factorization above yields
		\[
		(\iota_{q+1})_y(x)=t\cdot(\jmath_{q+1})_y(x)=t\cdot s=0,
		\]
		contradicting the injectivity of $(\iota_{q+1})_y$.
		As $y$ was arbitrary, $R^{q+1}f_*\mathscr{F}$ is torsion free.
	\end{proof}

	\medskip
	\noindent\textbf{Added in proof.} After this paper was completed, Y. Li--P{\u{a}}un--Wang--X. Zhou \cite[Theorem~1.2]{LPWZ26} gave a proof of  Conjecture \ref{siu-conj} under the additional assumption that the total space $X$ is K\"ahler. More precisely, if $p:X\to\Delta$ is a proper holomorphic submersion and $X$ is K\"ahler, then $P_m(X_t)$ is independent of $t\in\Delta$ for every $m\geq1$.


\begin{thebibliography}{LiPWZ26}
		\raggedright
		
		\bibitem[ACGH85]{ACGH85}
		E.~Arbarello, M.~Cornalba, P.~A. Griffiths, J.~Harris, \emph{Geometry of algebraic curves. Volume I}, Grundlehren der mathematischen Wissenschaften, vol.~267, Springer-Verlag, New York, 1985.
		
		\bibitem[At68]{Artin1}
		M.~Artin, \emph{On the solutions of analytic equations}, Invent. Math. 5 (1968), 277--291.
		
		\bibitem[BS76]{BS76}
		C.~B\u{a}nic\u{a}, O.~St\u{a}n\u{a}\c{s}il\u{a}, \emph{Algebraic methods in the global theory of complex spaces}, Editura Academiei, Bucharest; John Wiley \& Sons, London-New York-Sydney, 1976, translated from the Romanian.
		
		\bibitem[BCHM10]{BCHM}
		C.~Birkar, P.~Cascini, C.~D. Hacon, J.~M\textsuperscript{c}Kernan, \emph{Existence of minimal models for varieties of log general type}, J. Amer. Math. Soc. 23 (2010), no.~2, 405--468.
		
		\bibitem[BH93]{BH93}
		W.~Bruns, J.~Herzog, \emph{Cohen--Macaulay rings}, Cambridge Studies in Advanced Mathematics, vol.~39, Cambridge University Press, Cambridge, 1993.
		
		\bibitem[CDM17]{CDM17}
		Junyan Cao, J.-P.~Demailly, S.~Matsumura, \emph{A general extension theorem for cohomology classes on non reduced analytic subspaces}, Sci. China Math. 60 (2017), no.~6, 949--962.
		
		
		\bibitem[CRT26]{CRT25}
		Jian Chen, Sheng Rao, I-Hsun Tsai, \emph{Characterization of fiberwise bimeromorphism and specialization of bimeromorphic types I: locally Moishezon case}, preprint, 2026, \href{https://arxiv.org/abs/2506.12670v2}{arXiv:2506.12670v2}.
		
		\bibitem[CH01]{CH01}
		Jungkai Alfred Chen, C.~D. Hacon, \emph{Pluricanonical maps of varieties of maximal {A}lbanese dimension}, Math. Ann. 320 (2001), no.~2, 367--380.
		
		\bibitem[DH24]{DHAlbanese}
		O.~Das, C.~D. Hacon, \emph{Existence of good minimal models for {K}\"ahler varieties of maximal {A}lbanese dimension}, C. R. Math. Acad. Sci. Paris 362 (2024), 83--91.
		
		\bibitem[DHP24]{DHP23}
		O.~Das, C.~D. Hacon, M.~P{\u{a}}un, \emph{On the $4$-dimensional minimal model program for {K}\"ahler varieties}, Adv. Math. 443 (2024), Paper No.~109615, 68 pp.
		
		\bibitem[DO24]{DO24}
		O.~Das, Wenhao Ou, \emph{On the log abundance for compact {K}\"ahler threefolds}, Manuscripta Math. 173 (2024), nos.~1--2, 341--404.
		
		\bibitem[DO26]{DO26}
		O.~Das, Wenhao Ou, \emph{On the log abundance for compact {K}\"ahler threefolds {II}}, Proc. Lond. Math. Soc. (3) 132 (2026), no.~3, Paper No.~e70141, 63 pp.
		
		\bibitem[Dem12]{Dem12}
		J.-P.~Demailly, \emph{Analytic methods in algebraic geometry}, Surveys of Modern Mathematics, vol.~1, International Press, Somerville, MA; Higher Education Press, Beijing, 2012.
		
		\bibitem[DF04]{Dum04}
		D.~S. Dummit, R.~M. Foote, \emph{Abstract algebra}, third ed., John Wiley \& Sons, Inc., Hoboken, NJ, 2004.
		
		\bibitem[FS20]{FS20}
		S.~Filipazzi, R.~Svaldi, \emph{Invariance of plurigenera and boundedness for generalized pairs}, Mat. Contemp. 47 (2020), 114--150.
		
		\bibitem[Fj13]{Fujinoalb}
		O.~Fujino, \emph{On maximal Albanese dimensional varieties}, Proc. Japan Acad. Ser. A Math. Sci. 89 (2013), no.~8, 92--95.
		
		\bibitem[Gr74]{Grauert74}
		H.~Grauert, \emph{Der Satz von Kuranishi f\"ur kompakte komplexe R\"aume}, Invent. Math. 25 (1974), 107--142.
		
		\bibitem[GR84]{GR84}
		H.~Grauert, R.~Remmert, \emph{Coherent analytic sheaves}, Grundlehren der mathematischen Wissenschaften [Fundamental Principles of Mathematical Sciences], vol.~265, Springer-Verlag, Berlin, 1984.
		
		\bibitem[Gr65]{Gr65}
		P.~A. Griffiths, \emph{The extension problem for compact submanifolds of complex manifolds. {I}. {T}he case of a trivial normal bundle}, Proc. Conf. Complex Analysis (Minneapolis, 1964), Springer, Berlin--Heidelberg--New York, 1965, pp.~113--142.
		
		\bibitem[HLR26]{HLR26}
		C.~D. Hacon, Yi Li, Sheng Rao, \emph{On pseudo-effectivity and volumes of adjoint classes in K\"ahler families with projective central fiber}, preprint, 2026, \href{https://arxiv.org/abs/2602.04158v2}{arXiv:2602.04158v2}.
		
		\bibitem[HLX26]{HLX26}
		C.~D. Hacon, Yi Li, Lingyao Xie, \emph{Fujiki class $\mathcal{C}$ varieties and a K\"ahler criterion}, preprint, 2026, \href{https://arxiv.org/abs/2608.20588v1}{arXiv:2608.20588v1}.
		
		\bibitem[HMX13]{HMX13}
		C.~D. Hacon, J.~M\textsuperscript{c}\kern-0.06em Kernan, Chenyang Xu, \emph{On the birational automorphisms of varieties of general type}, Ann. of Math. (2) 177 (2013), no.~3, 1077--1111.
		
		\bibitem[HP24]{HP24}
		C.~D. Hacon, M.~P{\u{a}}un, \emph{On the canonical bundle formula and adjunction for generalized {K}\"ahler pairs}, preprint, 2024, \href{https://arxiv.org/abs/2404.12007v1}{arXiv:2404.12007v1}.
		
		\bibitem[HX26]{HX26}
		C.~D. Hacon, Lingyao Xie, \emph{On the {K}\"ahler {MMP} and the transcendental base-point-free theorem}, preprint, 2026, \href{https://arxiv.org/abs/2607.24986v1}{arXiv:2607.24986v1}.
		
		\bibitem[HM98]{HM98}
		J.~Harris, I.~Morrison, \emph{Moduli of curves}, Graduate Texts in Mathematics, vol.~187, Springer-Verlag, New York, 1998.
		
		\bibitem[Ha77]{Hart77}
		R.~Hartshorne, \emph{Algebraic geometry}, Graduate Texts in Mathematics, vol.~52, Springer-Verlag, New York--Heidelberg, 1977.
		
		\bibitem[Iv86]{Ive86}
		B.~Iversen, \emph{Cohomology of sheaves}, Universitext, Springer-Verlag, Berlin, 1986.
		
		\bibitem[Ka99a]{KawaDeform99}
		Y.~Kawamata, \emph{Deformations of canonical singularities}, J. Amer. Math. Soc. 12 (1999), no.~1, 85--92.
		
		\bibitem[KKMSD73]{KKMSD73}
		G.~Kempf, F.~Knudsen, D.~Mumford, B.~Saint-Donat, \emph{Toroidal embeddings I}, Lecture Notes in Mathematics, vol.~339, Springer-Verlag, Berlin--New York, 1973.
		
		\bibitem[KS58]{KS58}
		K.~Kodaira, D.~C. Spencer, \emph{On deformations of complex analytic structures. {I}, {II}}, Ann. of Math. (2) 67 (1958), 328--466.
		
		\bibitem[KS60]{KS60}
		K.~Kodaira, D.~C. Spencer, \emph{On deformations of complex analytic structures. {III}. {S}tability theorems for complex structures}, Ann. of Math. (2) 71 (1960), 43--76.
		
		\bibitem[Ko86]{Kollar86}
		J.~Koll\'ar, \emph{Higher direct images of dualizing sheaves. {I}}, Ann. of
		Math. (2) 123 (1986), no.~1, 11--42.
		
		\bibitem[Kol96]{Kol96}
		J.~Koll{\'a}r, \emph{Rational curves on algebraic varieties}, Ergebnisse der Mathematik und ihrer Grenzgebiete. 3. Folge. A Series of Modern Surveys in Mathematics, vol.~32, Springer-Verlag, Berlin, 1996.
		
		\bibitem[Ko21]{DefGeneralType}
		J.~Koll{\'a}r, \emph{Deformations of varieties of general type}, Milan J. Math. 89 (2021), 345--354.
		
		\bibitem[Ko23]{Kolbook}
		J.~Koll\'{a}r, \emph{Families of varieties of general type}, Cambridge Tracts in Mathematics, vol.~231, Cambridge University Press, Cambridge, 2023, with the collaboration of K.~Altmann and S.~J. Kov\'{a}cs.
		
		\bibitem[KM92]{KM92}
		J.~Koll\'{a}r, S.~Mori, \emph{Classification of three-dimensional flips}, J. Amer. Math. Soc. 5 (1992), no.~3, 533--703.
		
		
		\bibitem[LPWZ26]{LPWZ26}
		Yinji Li, M.~P{\u{a}}un, Zhiwei Wang, Xiangyu Zhou, \emph{Psh variation of singular K\"ahler--Einstein metrics and invariance of plurigenera for smooth K\"ahler families}, preprint, 2026, \href{https://arxiv.org/abs/2610.07826v1}{arXiv:2610.07826v1}.
		
		\bibitem[LW26]{LW26}
		Zhan Li, Zhiwei Wang, \emph{Invariance of plurigenera for generalized polarized pairs with abundant nef parts}, Complex Anal. Synerg. 12 (2026), no.~1, Paper No.~2, 17 pp.
		
		\bibitem[Ma89]{Mat89}
		H.~Matsumura, \emph{Commutative ring theory}, second ed., Cambridge Studies in Advanced Mathematics, vol.~8, Cambridge University Press, Cambridge, 1989, translated from the Japanese by M. Reid.
		
		\bibitem[Ma22]{Matsumura22}
		S.~Matsumura, \emph{Injectivity theorems with multiplier ideal sheaves for higher direct images under K\"ahler morphisms}, Algebr. Geom. 9 (2022), no.~2, 122--158.
		
		\bibitem[Mu81]{Mukai81}
		S.~Mukai, \emph{Duality between $D(X)$ and $D(\widehat{X})$ with its application to Picard sheaves}, Nagoya Math. J. 81 (1981), 153--175.
		
		\bibitem[Mu70]{Mum70}
		D.~Mumford, \emph{Abelian varieties}, Tata Institute of Fundamental Research Studies in Mathematics, vol.~5, Oxford University Press, London, 1970.
		
		\bibitem[Na90]{Nad90}
		A.~M. Nadel, \emph{Multiplier ideal sheaves and {K\"ahler--Einstein} metrics of positive scalar curvature}, Ann. of Math. (2) 132 (1990), no.~3, 549--596.
		
		\bibitem[Na86]{Naka85}
		N.~Nakayama, \emph{Invariance of the plurigenera of algebraic varieties under minimal model conjectures}, Topology 25 (1986), no.~2, 237--251.
		
		\bibitem[Na87]{Naka87}
		N.~Nakayama, \emph{The lower semicontinuity of the plurigenera of complex varieties}, Algebraic geometry, {S}endai, 1985, Adv. Stud. Pure Math., vol.~10, North-Holland, Amsterdam, 1987, pp. 551--590.
		
		\bibitem[Ny04]{Ny04}
		N.~Nakayama, \emph{Zariski-decomposition and abundance}, MSJ Memoirs, vol.~14, Mathematical Society of Japan, Tokyo, 2004.
		
		\bibitem[OT87]{OT87}
		T.~Ohsawa, K.~Takegoshi, \emph{On the extension of {$L^2$} holomorphic functions}, Math. Z. 195 (1987), 197--204.
		
		\bibitem[Pa07]{Paun07}
		M.~P{\u{a}}un, \emph{Siu's invariance of plurigenera: a one-tower proof}, J. Differential Geom. 76 (2007), no.~3, 485--493.
		
		\bibitem[Ra26]{Ra26}
		Sheng Rao, \emph{Upper semicontinuity of algebraic dimension}, preprint, 2024, revised 2026, with an appendix joint with Mu-Lin Li, Kai Wang, Mengjiao Wang, \href{https://arxiv.org/pdf/2407.02022}{arXiv:2407.02022v3}.
		
		\bibitem[RT21]{RT1}
		Sheng Rao, I-Hsun Tsai, \emph{Deformation limit and bimeromorphic embedding of Moishezon manifolds}, Commun. Contemp. Math. 23 (2021), no.~8, Paper No.~2050087.
		
		\bibitem[RT22]{RT2}
		Sheng Rao, I-Hsun Tsai, \emph{Invariance of plurigenera and {C}how-type lemma}, Asian J. Math. 26 (2022), no.~4, 507--554.
		
		\bibitem[Re94]{Remmert94}
		R.~Remmert, \emph{Local theory of complex spaces}, in \emph{Several complex variables VII: Sheaf-theoretical methods in complex analysis}, Encyclopaedia of Mathematical Sciences, vol.~74, Springer-Verlag, Berlin, 1994, pp. 7--96.
		
		
		\bibitem[Si98]{Siu98}
		Yum-Tong Siu, \emph{Invariance of plurigenera}, Invent. Math. 134 (1998), no.~3, 661--673.
		
		\bibitem[Si02]{Siu02}
		Yum-Tong Siu, \emph{Extension of twisted pluricanonical sections with plurisubharmonic weight and invariance of semipositively twisted plurigenera for manifolds not necessarily of general type}, Complex geometry ({G}\"{o}ttingen, 2000), Springer, Berlin, 2002, pp. 223--277.
		
		\bibitem[Sp95]{Spanier95}
		E.~H. Spanier, \emph{Algebraic topology}, Springer-Verlag, New York, [1995], corrected reprint of the 1966 original.
		
		\bibitem[Ta07]{Taka06}
		S.~Takayama, \emph{On the invariance and the lower semi-continuity of plurigenera of algebraic varieties}, J. Algebraic Geom. 16 (2007), no.~1, 1--18.
		
		\bibitem[Te15]{Teleman15}
		A.~Teleman, \emph{On the torsion of the first direct image of a locally free sheaf}, Ann. Inst. Fourier (Grenoble) 65 (2015), no.~1, 101--136.
		
		\bibitem[We94]{Weibel94}
		C.~A. Weibel, \emph{An introduction to homological algebra}, Cambridge Studies in Advanced Mathematics, vol.~38, Cambridge University Press, Cambridge, 1994.
		
		\bibitem[ZZ22]{ZZ22}
		Xiangyu Zhou, Langfeng Zhu, \emph{Extension of cohomology classes and holomorphic sections defined on subvarieties}, J. Algebraic Geom. 31 (2022), no.~1, 137--179.
		
	\end{thebibliography}
\end{document}